\documentclass{amsart}

\usepackage{graphicx}
\usepackage{amscd,amsmath,amsopn,amssymb, amsthm,multicol} 
\usepackage{hyperref}
\usepackage[all]{xy}
\usepackage{amscd}
\usepackage{lscape}
\usepackage{slashed}
\usepackage{graphicx}
\usepackage{lscape}
 \usepackage{cancel}

\newtheorem{theorem}{Theorem}[section]
\newtheorem{lemma}[theorem]{Lemma}
\newtheorem{proposition}[theorem]{Proposition}
\newtheorem{corollary}[theorem]{Corollary}

\theoremstyle{definition}
\newtheorem{Definition}[theorem]{Definition}
\newtheorem{example}[theorem]{Example}

\theoremstyle{remark}
\newtheorem{remark}[theorem]{Remark}
\newtheorem{Question}[theorem]{Question}

\numberwithin{equation}{section}

 \input ulem.sty
 
\DeclareMathOperator{\Ad}{Ad}
\DeclareMathOperator{\ad}{ad}
\DeclareMathOperator{\Aut}{Aut}

\DeclareMathOperator{\Id}{Id}
\DeclareMathOperator{\tr}{tr}

\DeclareMathOperator{\Sym}{Sym}

\DeclareMathOperator{\Cl}{C\ell}
\DeclareMathOperator{\Ric}{Ric}
\DeclareMathOperator{\Sca}{Scal}
\DeclareMathOperator{\Hom}{Hom}
\DeclareMathOperator{\Cas}{Cas}

\DeclareMathOperator{\vol}{vol}
\DeclareMathOperator{\Span}{span}
\DeclareMathOperator{\Spin}{Spin}

\DeclareMathOperator{\Jac}{Jac}
\DeclareMathOperator{\Div}{div}
\newcommand{\fr}{\mathfrak}
\newcommand{\al}{\alpha}
\newcommand{\be}{\beta}

\newcommand{\bb}{\mathbb}
\newcommand{\cal}{\mathcal}

\DeclareMathOperator{\SO}{SO}
\DeclareMathOperator{\Sp}{Sp}
 \DeclareMathOperator{\SU}{SU}
\DeclareMathOperator{\U}{U}
\DeclareMathOperator{\G}{G}

\DeclareMathOperator{\Oo}{O}
\DeclareMathOperator{\E}{E}
\DeclareMathOperator{\Ss}{S}
\DeclareMathOperator{\Ed}{End}

\begin{document}   

 \title
{Invariant connections with skew-torsion and   $\nabla$-Einstein  manifolds}
\author{Ioannis Chrysikos}
 \address{Department of Mathematics and Statistics, Masaryk University, Brno  611 37, Czech Republic}
 \email{chrysikosi@math.muni.cz}




 

 \begin{abstract}
 For  a compact  connected Lie group $G$     we study the class of  bi-invariant affine connections whose geodesics through $e\in G$ are the 1-parameter subgroups.   We show  that  the bi-invariant affine connections which induce   derivations on the corresponding Lie algebra   $\fr{g}$  coincide  with the  bi-invariant metric connections.  Next we  describe the geometry of  a naturally reductive space  $(M=G/K, g)$  endowed with a family of  $G$-invariant connections $\nabla^{\al}$ whose torsion  is a multiple of the torsion of the canonical connection $\nabla^{c}$.  For the spheres $\Ss^{6}$ and $\Ss^{7}$ we prove that the space of $\G_2$  (resp. $\Spin(7)$)-invariant affine or metric connections consists of the family $\nabla^{\al}$.  Then we examine the ``constancy'' of the induced Ricci tensor $\Ric^{\al}$ and  prove that any  compact simply-connected isotropy irreducible  standard homogeneous  Riemannian manifold, which is not  a symmetric space of Type I,  is a $\nabla^{\al}$-Einstein manifold for any $\al\in\bb{R}$. We also provide examples of $\nabla^{\pm 1}$-Einstein structures for a class of  compact   homogeneous spaces $M=G/K$ with two isotropy summands.
 \end{abstract}
\maketitle  

    \section*{Introduction}
  Given a Riemannian manifold $(M^{n}, g)$,  metric connections whose torsion is a 3-form are   geometrically the connections which have the same geodesics  as the Levi-Civita
connection.   These connections play  a crucial role in the theory of non-integrable geometries and they admit physical applications   in type II string theory, see for example \cite{FrIv, Agr03, Srni}.    A very  remarkable example  is the so-called canonical connection $\nabla^{c}$ on a naturally reductive space.
   This article is a contribution to the geometry of  naturally reductive manifolds and  Lie groups  with respect to an invariant  metric  connection with skew-torsion.  Our approach   is fundamental and it mainly relies  on the homogeneous structure that such a manifold carries. 
     We begin by describing bi-invariant  affine connections on a compact connected   Lie group $G$. 
  Among the different bi-invariant   connections that one can consider on   $G$,   we are concerned  with those for which the Nomizu map $\Lambda : \fr{g}\to\fr{gl}(\fr{g})$ satisfies the property  $\Lambda(X)X=0$,   for any $X\in\fr{g}\cong T_{e}G$.  
  Although  any bi-invariant {\it metric} connection  has this property, the converse   is not necessarily true;  counterexamples are known for $G=\U(n)$ \cite{Laq1, AFH}.   It is therefore natural to ask what  conditions  we have to impose on the Nomizu map in order to establish a possible correspondence.    For this,   we propose a   formula which relates an equivariant derivation of  $\fr{g}$ with  the torsion and the curvature of   a bi-invariant connection on $G$, satisfying   $\Lambda(X)X=0$  $(X\in\fr{g})$. This  observation enables us to classify  the $\Ad(G)$-equivariant derivations $D  : \fr{g}\to{\rm Der}(\fr{g})$ carrying  the  property $D(X)X=0$. Then, we show that  the class of   bi-invariant  affine   connections  which induce derivations $\Lambda  : \fr{g}\to{\rm Der}(\fr{g})$ on the corresponding Lie algebra $\fr{g}$  coincides with the class of  bi-invariant metric  connections on $G$ (see Theorem   \ref{THM1}).

   After this description, a  natural step is the investigation of naturally reductive Riemannian manifolds  $(M=G/K, g)$ endowed with $G$-invariant connections    whose torsion  is a multiple of the torsion $T^{c}$  of the canonical connection $\nabla^{c}$, say $T^{\al}=\al\cdot T^{c}$ for some parameter $\al$.  
     For an    irreducible symmetric space $(M=G/K, g)$ of Type I  one can show that the space  of $G$-invariant {\it metric} connections consists only of the canonical connection $\nabla^{c}\equiv \nabla^{g}$  (see \cite[Thm.~2.1]{Laq2} and \cite[Rem.~3.2]{AFH}). 
   Here, we     primarily focus    on   symmetric spaces  which can be (re)presented as cosets of distinct  Lie groups, e.g. the spheres $\Ss^{6}$ and $\Ss^{7}$.  In terms of representation theory we show that  the space of $\Spin(7)$-invariant affine (or metric) connections on the 7-sphere $\Ss^{7}=\Spin(7)/\G_2$ is 1-dimensional; it consists of the family 
     $\{\nabla^{\al} : \al\in\bb{R}\}$ described above.   The same is true for the space of $\G_2$-invariant affine (or metric)  connections on the sphere $\Ss^{6}=\G_2/\SU(3)$, with the difference that    $\al\in\bb{C}$, i.e. there is a 2-dimensional family of $\G_2$-invariant affine   connections on $\Ss^{6}=\G_2/\SU(3)$, which has  the same geodesics with the Levi-Civita connection (see Theorem \ref{THM2}).    These invariant connections occur   since the cosets $\Spin(7)/\G_2$ and $\G_2/\SU(3)$, although diffeomorphic to  a symmetric space, do not provide us with symmetric pairs (see  \cite[7.107  Table 6]{Bes}).

 The rest of the article is a detailed study of   $\nabla$-Einstein structures with skew-torsion   on compact   naturally reductive spaces. Given a Riemannian manifold $(M^{n}, g)$ $(n\geq 3)$ equipped with a metric connection $\nabla$ with non-trivial skew-torsion $T\in\Lambda^{3}(T^{*}M)$, a {\it $\nabla$-Einstein structure with skew-torsion}, or in short  a {\it $\nabla$-Einstein structure}, is a  generalization of the Riemannian Einstein condition, given by  a tuple  $(M^{n}, g, \nabla, T)$   satisfying the equation $\Ric_{S}=(\Sca\cdot g)/n$.    Here, $\Ric_{S}$   denotes  the symmetric part of the Ricci tensor associated to $\nabla$  and $\Sca$ is  the corresponding scalar curvature. Solutions of this equation  naturally appear  in the context  of non-integrable geometries, where a metric connection different than the Levi-Civita connection is adapted to the geometry under consideration, the so-called characteristic connection \cite{FrIv, Srni}.   A variational principle has been recently deduced in \cite{AFer}, proving that $\nabla$-Einstein structures are optimal between the different metric connections with skew-torsion that one can define by choosing pairs  $(g, T)$  of Riemannian metrics and compatible totally skew-symmetric torsion tensors. In this paper,      we describe the $\nabla^{\al}$-Einstein equation on  {\it compact}  naturally reductive Riemannian manifolds  in terms of Casimir elements. 
     We prove that any  compact isotropy irreducible standard homogeneous Riemannian manifold $(M^{n}=G/K, g)$ of a compact connected semi-simple Lie group $G$, which is not a symmetric space of Type I,  is a $\nabla^{\al}$-Einstein manifold for {\it any} $\al\in\bb{R}$  
       (see Theorem \ref{THM4}).  
        Notice   that symmetric spaces of Type I are never $\nabla^{\al}$-Einstein (since $\nabla^{\al}\equiv\nabla^{c}\equiv\nabla^{g}$), in contrast to  symmetric spaces of Type II, i.e. compact simple  Lie groups which  are $\nabla^{\al}$-Einstein with parallel torsion for any $\al\in\bb{R}$, with  the well-known flat $\pm 1$-connections of Cartan-Schouten being the trivial members (see \cite[Lemma 1.8]{AFer} and Example \ref{LIEGROUP}). 
     
       In the final section  we extend our study on  compact        homogeneous Riemannian manifolds $M=G/K$ of a  compact  connected semi-simple Lie group $G$, whose isotropy representation decomposes into two non-trivial   irreducible and inequivalent $K$-submodules, such that
 \begin{equation}\label{incl}
 [\fr{k}, \fr{m}_{i}]\subset\fr{m}_{i},\quad     [\fr{m}_{1}, \fr{m}_{1}]\subset\fr{k}\oplus\fr{m}_{2}, \quad   [\fr{m}_{1}, \fr{m}_{2}]\subset\fr{m}_{1}, \quad  [\fr{m}_{2}, \fr{m}_{2}]\subset\fr{k}.
\end{equation}
 As a first step, we characterize the invariant  metric  connections on   $T(G/K)$ which   have skew-torsion with respect to a 1-parameter family of $G$-invariant metrics $\{g_{t} : t\in\bb{R}_{+}\}$. They  exist only for the Killing metric $t=1/2$, under the further assumption that the associated Nomizu map $\Lambda_{\fr{m}} : \fr{m}\to\fr{so}(\fr{m})$ satisfies $\Lambda_{\fr{m}}(X)X=0$, for any $X\in\fr{m}$ (see Theorem \ref{chato}). 
     Based on this  characterization,  we introduce a  new 2-parameter family of $G$-invariant metric connections, say $\nabla^{s, t}$ with $s\in\bb{R}$ and $t>0$, which  joins the connections $\nabla^{t}$ and $\nabla^{c}$,   and for the  Killing metric gives rise to a family of invariant connections with skew-torsion, namely $\{\nabla^{s, \frac{1}{2}} : s\in\bb{R}\}$.  For completeness,  we examine  the full algebraic type of the    torsion $T^{s, t}$; we show (even by a theory based on Dirac operators) that it does not contain any component of vectorial type. Then we describe   the $\nabla^{s, \frac{1}{2}}$-Einstein condition   in terms of the  Casimir eigenvalues $\Cas_1$ and $\Cas_2$.  We prove that $M=G/K$ is a $\nabla^{s}$-Einstein manifold with skew-torsion  for the values $s=0$ or $s=2$ if and only if    the Killing metric $g_{B}\equiv g_{1/2}$ is Einstein, i.e. $\Cas_1=\Cas_2$ (see Theorem \ref{THM5}).   
Finally, we use results of  a previous work \cite{Chry1}  related to  author's  PhD thesis to provide a series of examples  of  homogeneous spaces  carrying $\nabla^{s}$-Einstein structures with skew-torsion,   for  the values  $s=0, 2$. These are flag manifolds of a compact connected simple Lie group $G$ and they are the first known examples of {\it infinite families}  of {\it non-isotropy irreducible} homogeneous Riemannian manifolds, admitting  $\nabla$-Einstein structures with  skew-torsion.  

\smallskip
\noindent {\bf Acknowledgements:} This work  has been fully supported   
  by Masaryk University under the Grant Agency of Czech Republic, project no.14-2464P. The author warmly acknowledges  I.~Agricola, D.~V.~Alekseevsky,  A.~Arvanitoyeorgos, S.~Chiossi,  Th. Friedrich and Y.~Sakane for several discussions and remarks. 


\section{Homogeneous Riemannian manifolds and invariant connections}  \label{prel}
\subsection{$G$-invariant connections}   Consider a connected homogeneous  Riemannian manifold $(M=G/K, g)$,  where $G\subset I(M)$ is a  closed subgroup of the isometry group   and $K$ is the isotropy subgroup of some point of $M$.    Assume for simplicity that the transitive $G$-action is effective and $K$ is connected. We shall denote by $\fr{g}, \fr{k}$ the corresponding Lie algebras.  Because $K$ is compact, one can always fix an $\Ad(K)$-invariant  splitting $\fr{g}=\fr{k}\oplus\fr{m}$, i.e. $[\fr{k}, \fr{m}]\subset\fr{m}$. Then, $\fr{m}$ is identified with the tangent space $T_{o}M$ of $M$  $(o=eK\in M)$ and  the  isotropy  representation $\chi : K\to \SO(\fr{m})\subset\Aut(\fr{m})$   coincides with the restriction of the adjoint representation $\Ad|_{K}$ on $\fr{m}$, see \cite{Bes}.
Let us denote by  $\fr{aff}_{G}(P)$ the space of $G$-invariant affine connections on a homogeneous principal bundle $P\to G/K$ over $M=G/K$. Let also   ${\rm Hom}_{K}^{\bb{R}}(\fr{m}\otimes\fr{m}, \fr{m})$  be the space of $K$-intertwining maps $\fr{m}\otimes\fr{m}\to\fr{m}$.   By a theorem of H.~C.~Wang \cite{W} it is well-known that a linear  $G$-invariant  connection $\nabla : \Gamma(TM)\to\Gamma(T^{*}M\otimes TM)$    is described by a $\bb{R}$-linear map 
$\Lambda_{\fr{m}}  : \fr{m}\to\fr{gl}(\fr{m})$ which is equivariant under the isotropy representation, i.e. $\Lambda_{\fr{m}}(\Ad(k)X)=\Ad(k)\Lambda_{\fr{m}}(X)\Ad(k)^{-1}$ for any $ X\in\fr{m}$ and $k\in K$.  Writing  $\Lambda_{\fr{m}}(X)Y=\eta(X, Y)$ for some $\Ad(K)$-equivariant bilinear map  $\eta : \fr{m}\times\fr{m}\to\fr{m}$, i.e. $\eta(\Ad(k)X, \Ad(k)Y)=\Ad(k)\eta(X, Y)$ for any $X, Y\in\fr{m}$ and $k\in K$, one   can finally establish the identification (see \cite{Cap, Laq2})
 \[
 \fr{aff}_{G}\big(F(G/K)\big)\cong {\rm Hom}_{K}^{\bb{R}}(\fr{m}\otimes\fr{m}, \fr{m}).
 \]
 The linear map $\Lambda_{\fr{m}}$ is commonly referred to us as the {\it Nomizu map} (for details see \cite{A, Kob2})  and it nicely describes  the properties of $\nabla$.    For example,  pull back the Riemannian metric $g:=\langle \ , \ \rangle_{o}$ on $T_{o}M$ to an $\Ad(K)$-invariant inner product $\langle \ , \ \rangle$ on $\fr{m}$.  Then, $\nabla$ is metric, i.e.   $\Lambda_{\fr{m}}(X)$ lies in $\fr{so}(\fr{m})$ for any $X\in\fr{m}$  if and only if $\langle \Lambda_{\fr{m}}(X)Y, Z\rangle+\langle Y, \Lambda_{\fr{m}}(X)Z\rangle=0$ for any $X, Y, Z\in\fr{m}$.
  Furthermore, the torsion   and   curvature    are given by 
  \begin{equation}\label{torsion}
 \left.\begin{tabular}l
 $T(X, Y)_{o}=\Lambda_{\fr{m}}(X)Y-\Lambda_{\fr{m}}(Y)X-[X, Y]_{\fr{m}}$ \\
 $R(X, Y)_{o}= [\Lambda_{\fr{m}}(X), \Lambda_{\fr{m}}(Y)]-\Lambda_{\fr{m}}([X, Y]_{\fr{m}})-\ad([X, Y]_{\fr{k}})$  \\
 \end{tabular}\right\}.
    \end{equation} 
Viewing the torsion as a  $(3, 0)$-tensor  $T(X, Y, Z):=\langle T(X, Y), Z\rangle$  we will call $T$  the {\it torsion form}   if and only if it is skew-symmetric in $Y$ and $Z$ (and hence totally skew-symmetric).
    
 Recall   that  $M=G/K$ carries a distinguished invariant connection, the so-called  {\it canonical connection} $\nabla^{c}$. This is the unique $G$-invariant connection whose Nomizu map $\Lambda_{\fr{m}} : \fr{m}\to\fr{so}(\fr{m})$  is the zero map, i.e. $\Lambda_{\fr{m}}(X)=0$,  for any $X\in\fr{m}$  \cite{A, Kob2}.  The canonical connection   depends on  the choice of $\fr{m}\cong T_{o}M$, for example its torsion is given by  $T^{c}(X, Y)=-[X, Y]_{\fr{m}}$.  Moreover, any $G$-invariant tensor field is $\nabla^{c}$-parallel, in particular $\nabla^{c}T^{c}=0=\nabla^{c}R^{c}$.  A homogeneous Riemannian manifold  $(M=G/K, g)$ is called  {\it naturally reductive}  with respect to $G$ if and only if  the torsion of  $\nabla^{c}$ is a 3-form on $\fr{m}$, i.e.   $T^{c}(X, Y, Z):=\langle T^{c}(X, Y), Z\rangle\in\Lambda^{3}(\fr{m})$ for any $X, Y, Z\in\fr{m}$.  
          Geometrically, the  notion of  natural reductivity is equivalent to say that for each vector $X\in\fr{m}$ the orbit $\gamma(t):=\exp(tX)o$ is a geodesic on $M$, which means the Riemannian geodesics coincide with the $\nabla^{c}$-geodesics. 

\subsection{Bi-invariant connections}  A compact connected Lie group  $M=G$  with a bi-invariant metric $\rho$  can be viewed as a symmetric space of the form $(G\times G)/\Delta G$. The  Cartan decomposition is given by  $\fr{g}\oplus\fr{g}=\Delta\fr{g}\oplus\fr{p}$, where both $\Delta\fr{g}:=\{(X, X)\in\fr{g}\oplus\fr{g} : X\in\fr{g}\}$ and $\fr{p}:=\{(X, -X)\in\fr{g}\oplus\fr{g}  : X\in\fr{g}\}$  are isomorphic to $\fr{g}$, as $G$-modules. The isotropy representation is   the adjoint representation of $G$, i.e. $\chi(g, g)(X, -X):=(\Ad(g)X, -\Ad(g)X)$. Hence,  as a symmetric space, $G$ is isotropy irreducible  if and only if  $G$ is simple.  In this note we are interested in  {\it bi-invariant}   connections on $G$, i.e.   $(G\times G)$-invariant   affine connections.   Such a connection, say $\nabla^{\eta}$, is completely  described by a bilinear map  $\eta : \fr{g}\times\fr{g}\to\fr{g}$  such that $\eta(\Ad(g)X, \Ad(g)Y)=\Ad(g)\eta(X, Y)$, for any $g\in G$ and  $X, Y\in\fr{g}$    \cite{Laq1}.  The  associated Nomizu map $\Lambda^{\eta} : \fr{g}\to\Ed(\fr{g})$ is given by $\Lambda^{\eta}(X)Y:=\eta(X, Y)$ and the equivariant condition is expressed by $\Lambda^{\eta}(\Ad(g)X)=\Ad(g)\Lambda^{\eta}(X)\Ad(g)^{-1}$.   For a  (compact) {\it simple} Lie group $G$ there exists a 1-dimensional family of   {\it canonical connections}  which  joins the Levi-Civita connection with the   flat $\pm 1$-connections of Cartan-Schouten (see \cite[Rem.~6.1]{Olmos} or \cite[p.~18]{AFH}).  To be more precise,  it is induced from   the reductive decomposition $\fr{g}\oplus\fr{g}=\Delta\fr{g}\oplus\fr{p}_{\al}$, with  $\al\in\bb{R}$ and 
  \[
  \fr{p}_{\al}:=\{X_{\al}:=\big(\frac{\al+1}{2}X, \frac{\al-1}{2}X\big)\in\fr{g}\oplus\fr{g} : X\in\fr{g}\}\cong\fr{g},
  \]
for example. Then, one computes $[X_{\al}, Y_{\al}]_{\fr{p}_{\al}}
=\al\big(\frac{\al+1}{2}[X, Y], \frac{\al-1}{2}[X, Y]\big)$ and   hence the torsion of the induced connection     is given by $T^{\al}(X, Y):=-[X_{\al}, Y_{\al}]_{\fr{p}_{\al}}=-\al [X, Y]$, for any  $X, Y\in\fr{g}$. To summarise:
  
 \begin{theorem}\label{cc}
 On   a  compact connected simple Lie group $G\cong (G\times G)/\Delta G$  endowed with a bi-invariant metric $\rho$,      there exists a 1-dimensional family    of  bi-invariant {\it canonical}   connections, namely $\nabla^{\al}_{X}Y=\eta^{\al}(X, Y)=\frac{1-\al}{2}[X, Y]$ $(\al\in\bb{R})$  (up to   scale).  The   curvature has the form $R^{\al}(X, Y)Z=(1-\al^{2})[Z, [X, Y]]/4$ for any $X, Y, Z\in\fr{g}$.  Thus,   $(G, \rho)$ endowed with one of the connections $\nabla^{\pm 1}$ becomes    flat, i.e. $R^{\pm 1}\equiv 0$. Moreover, the   torsion $T^{\al}$ is $\nabla^{\al}$-parallel  for any $\al\in\bb{R}$  (by the Jacobi identity).
 \end{theorem}

  \section{Metric bi-invariant  connections and derivations}\label{CLieG}
    We    recall     the classification of {\it metric} bi-invariant connections on a {\it compact} Lie group $G$ by \cite{AFH}.  For the sake of completeness, and since we will use this result, we explain the main idea of the proof (adapted in our notation).  This is    essentially based on the classification of bi-invariant affine connections given in  \cite{Laq1}. 
    \begin{theorem}\textnormal{(\cite[Thm.~3.1]{AFH})}\label{AFHoll}
Consider  a compact connected Lie group  $G$   with a bi-invariant metric $\rho$ and   let $\fr{g}=\fr{g}_{0}\oplus\fr{g}_{1}\oplus\cdots\oplus\fr{g}_{r}$ be   the decomposition of the Lie algebra $\fr{g}=T_{e}G$ into its centre $\fr{g}_{0}$ and simple ideals $\fr{g}_{i}$ $(1\leq i\leq  r)$.  Then, a  bi-invariant metric connection on $G$   is given by  (up to scale)   \begin{equation}\label{thec}
  \nabla^{\al}_{X}Y:=\eta^{\al}(X, Y)=\sum_{1\leq i\leq r}((1-\al_{i})/2) \cdot  [X, Y]_{\fr{g}_{i}},  
    \end{equation}
for any $X, Y\in\fr{g}$, where $\al:=(\al_1, \ldots, \al_{r})\in\bb{R}^{r}$.
The torsion and the curvature of this $r$-parameter family are   $T^{\al}(X, Y)=-\sum_{1\leq i\leq r}\al_{i}\cdot [X, Y]_{\fr{g}_{i}}$ and $R^{\al}(X, Y)Z=\sum_{1\leq i\leq r}((1-\al_{i}^{2})/4)\cdot [Z, [X, Y]_{\fr{g}_{i}}]_{\fr{g}_{i}}$, respectively.
  \end{theorem}
 \begin{proof} 
Consider first a bilinear $\Ad(G)$-equivariant map $\eta : \fr{g}\times\fr{g}\to\fr{g}$ corresponding to  a bi-invariant metric connection   $\nabla$ on $G$. Since $\nabla$ is metric with respect to $\rho$, $\eta$ is skew-symmetric with respect to the induced $\Ad(G)$-invariant inner product $\langle \ , \ \rangle$, i.e.   $\eta_{X}:=\Lambda(X)\in\fr{so}(\fr{g})$, for any $X\in\fr{g}$.   Thus,  
a bi-invariant metric connection $\nabla$ has skew-torsion $T\in \Lambda^{3}(\fr{g})$ if and only if $\eta(X, X)=\Lambda(X)X=0$ for any $X\in\fr{g}$  
 and this corrects a small   error in \cite[Lem.~3.1]{AFH}, see also \cite[Lem.~2.1]{Agr03}.
   Obviously,  a connection  induced by the adjoint representation of $\fr{g}$  verifies  this condition.  Hence, the most interesting part of the proof is that of uniqueness. We breake the argument up into  two steps.

\noindent {\it  1st Step:}  We begin with the additional assumption that  $G$ is  simple.  By Theorem \ref{cc} we know that the bilinear map    $\lambda(X, Y)=(1-\al)[X, Y]/2$  defines   a 1-dimensional family of bi-invariant metric connections on  $G$ with  torsion  $T^{\al}(X, Y)=-\al[X, Y]$,  for any $X, Y\in\fr{g}$.  We need now to show that this is the unique family (up to scale).  
The space of   bi-invariant affine connections on $G$  is isomorphic to the space $\Hom_{G}^{\bb{R}}(\fr{g}\otimes\fr{g}, \fr{g})$. 
Since $\fr{g}$ is  irreducible (and of real type),  it is sufficient to compute    the multiplicity of $\fr{g}$ inside $\fr{g}\otimes\fr{g}={\rm S}^{2}(\fr{g})\oplus \Lambda^{2}(\fr{g})$. 
In \cite{Laq1}  H.~T.~Laquer confirms  that  for any compact simple Lie group the multiplicity of $\fr{g}$ in $\Lambda^{2}(\fr{g})$ is one and only for $\SU(n)$ ($n\geq 3$) there is     a {\it new} copy of $\fr{g}$  inside ${\rm S}^{2}(\fr{g})$ with the same  multiplicity. Thus, for $G$ simple with $G\ncong\SU(n)$ $(n\geq 3)$,     the   unique family of  bi-invariant affine connections  is determined  by the bilinear map $\lambda$ (up to  scale).  For $\SU(n)$ the ``exceptional'' family corresponds to the symmetric bilinear map $\eta^{\rm exc}(X, Y)=i(XY+YX-(2/n){\rm tr}(XY)\cdot I)$, where $I$ is the $n\times n$ identity matrix.  However,  the induced affine   connection  is not  metric with respect  to a bi-invariant metric, e.g. the negative of the Killing form \cite{AFH}. This proves the claim.

\noindent {\it  2nd Step:}   Let us  drop now the latter condition and explain   the more general case of a compact   Lie group $G$.  Consider the decomposition of the corresponding Lie algebra $\fr{g}=T_{e}G$ into its centre and simple ideals $\fr{g}=\fr{g}_{0}\oplus\fr{g}_{1}\oplus\cdots\oplus\fr{g}_{r}$ and   write $X=X_{0}+X_{1}+\cdots+X_{r}$.  
 For any simple ideal  $\fr{g}_{i}$ one can apply the method described in the first step, by using  the bi-invariant connection $\nabla^{\al_{i}}$ induced by the bilinear map  $\eta^{\al_{i}}(X, Y)=((1-\al_{i})/2) \cdot  [X, Y]_{\fr{g}_{i}}$  for some $\al_{i}\in\bb{R}$, where $[X, Y]_{\fr{g}_{i}}:=[X_{i}, Y_{i}]$.  Obviously, $\nabla^{\al_{i}}$ is metric  with respect to the restriction $\langle \ , \ \rangle|_{\fr{g}_{i}}=x_{i}\rho|_{\fr{g}_{i}}$,  where  $x_{i}$ are real positive numbers for any $i=1, \ldots, r$.    Consider now some scalar product $b$ on the centre $\fr{g}_{0}$ and notice that $\eta^{\al_{0}}\equiv 0$.  The $\Ad(G)$-invariant scalar product $\langle \ ,  \ \rangle$ can be expressed by $\langle \ , \ \rangle = b|_{\fr{g}_{0}}+x_{1}\rho|_{\fr{g}_{1}}+\cdots +x_{r}\rho|_{\fr{g}_{r}}$ for some $x_{i}\in\bb{R}_{+}$. 
Hence, it is not difficult to see that the map defined by $\eta^{\al}(X, Y):=\sum_{i=1}^{r}\eta^{\al_{i}}(X, Y)$ with $\al:=(\al_1, \ldots, \al_{r})$, induces  a family of bi-invariant  connections on $G$  which are metric with respect to $\langle \ ,  \ \rangle$.  The  associated torsion is given by $T^{\al}(X, Y)=-\sum_{i=1}^{r}\al_{i}\cdot [X, Y]_{\fr{g}_{i}}$ and due to the $\ad(\fr{g})$-invariance of $\langle \ , \ \rangle$, the induced 3-tensor is   a 3-form on $\fr{g}$.  On the other hand, by \cite{Laq1} it is known that  besides $\SU(n)$ $(n\geq 3)$, only for $\U(n)$   $(n\geq 2)$  one can construct  affine  bi-invariant connections corresponding to $\Ad(\U(n))$-equivariant bilinear maps different from  the   Lie bracket (for details see   \cite[Thm. 9.1]{Laq1} and Proposition \ref{CLG}).  However,  as for $\SU(n)$, in \cite[Thm 3.1]{AFH} it was explained that the induced connections fail to carry the metric property.    Using now   \cite[Thm. 9.1]{Laq1}, we conlcude that for an  arbitrary compact Lie  group $G$  a  bi-invariant metric connection necessarily  corresponds to a copy of $\fr{g}$  inside $\Lambda^{2}(\fr{g})$,  and this is given by (\ref{thec}) (up to  scale).  
  \end{proof}
\begin{corollary}\label{onedir}
Any bi-invariant metric connection $\nabla$ on  a compact connected Lie group $G$ endowed with a bi-invariant metric, has (totally) skew-symmetric torsion $T\in\Lambda^{3}(\fr{g})$.
\end{corollary}
 
    \begin{Question}\label{q1}
 Let $G$ be a compact connected Lie group with Lie algebra $\fr{g}$.  Given    an arbitrary  bi-invariant {\it affine} connection $\nabla$  whose Nomizu map $\Lambda : \fr{g}\to\Ed(\fr{g})$ satisfies  the equation
  \begin{equation}\label{stc}
   \Lambda(X)X=0, \quad \forall \ X\in\fr{g},
   \end{equation}
   is it  true that $\nabla$ is metric with respect to a bi-invariant metric?  In other words, are the conditions $\Lambda(X)X=0$ and $\Lambda(X)\in\fr{so}(\fr{g})$ equivalent  for {\it any} bi-invariant {\it affine} connection on $G$?
 \end{Question}
\noindent A bi-invariant connection  satisfying   (\ref{stc})   has as geodesics orbits of  the one-parameter subgroups of $G$ (in the simple case the same geodesics   with the 1-parameter family of  canonical connections on $G$,  see \cite[Prop.~2.9, Ch.~X]{Kob2}).  Hence, as we explained before, if  $\nabla$  is metric with respect to a bi-invariant metric on $G$  then its torsion must be a 3-form on $\fr{g}$, i.e. $\Lambda(X)X=0$.  However, the ``converse'' is not true, i.e.  the  previous question admits a {\it negative} answer with counterexamples appearing  for $\U(n)$ (see  \cite{Laq1, AFH} and for   details the proof of  Proposition \ref{CLG}). Hence,  we ask:  
 \begin{Question}\label{q2}
 Which  further conditions do we have to impose on the Nomizu map   $\Lambda : \fr{g}\to\Ed(\fr{g})$ of a    bi-invariant {\it affine} connection on $G$ satisfying  (\ref{stc})  in order to be metric with respect to a bi-invariant metric? In  other words,  which subclass of bi-invariant affine connections satisfying (\ref{stc}) can be identified with the   class of   bi-invariant {\it metric}  connections on $G$?
 \end{Question}
\noindent     Our  answer relates  the {\it flat connections} of this type, which coincide with the    $\pm 1$-connections   discussed in Theorem \ref{cc}.  We   should emphasize once more that here we drop the condition that $G$ is simple.
    \begin{lemma}\label{REPRES}
 Let $G$  be a compact connected Lie group and let $\nabla$ be a    bi-invariant affine  connection with $\Lambda(X)X=0$, for any $X\in\fr{g}$. Then the following are equivalent:
 \begin{enumerate}
 \item[$(a)$] $\nabla$ is flat  $R\equiv0$, i.e. $\Lambda(X) : \fr{g}\to\fr{g}$ is a representation  for any $X\in\fr{g}$.
\item[$(b)$]  $\Lambda(X)=\ad(X)$,   or $\Lambda(X)=0$ for any $X\in\fr{g}$, and these  are the unique   bi-invariant linear connections which satisfy $(a)$.
\end{enumerate}
  \end{lemma}
  \begin{proof}
  By definition, $R(X, Y)=[\Lambda(X), \Lambda(Y)]-\Lambda([X, Y])$ and thus $\nabla$ is flat $R\equiv0$  if an only if    $\Lambda : \fr{g}\to\fr{gl}(\fr{g})$ is a representation (for example, the Riemannian connection  does not induce  a representation).   Assume   that $R\equiv 0$, i.e.  $\Lambda(X)\Lambda(Y)Z-\Lambda(Y)\Lambda(X)Z-\Lambda([X, Y])Z=0$, for any $X, Y, Z\in\fr{g}$.  By polarization, equation (\ref{stc}) is equivalent to    $\Lambda(X)Y+\Lambda(Y)X=0$ for any $X, Y\in\fr{g}$. Thus    $\Lambda(Y)[X, Y]=-\Lambda([X, Y])Y$  and   after setting $Y=Z$  inside the equation $R\equiv 0$,  it follows that
  \begin{eqnarray*}
  0&=&-\Lambda(Y)\Lambda(X)Y-\Lambda([X, Y])Y= -\Lambda(Y)\Lambda(X)Y+\Lambda(Y)[X, Y]=-\Lambda(Y)\big(\Lambda(X)Y-[X, Y]\big).
  \end{eqnarray*}
   Therefore $\Lambda(X)=\ad(X)$ or $\Lambda(X)=0$, for any $X\in\fr{g}$. The converse is trivial.  
           \end{proof}

\noindent  From now on we shall denote the  special connections presented in Lemma \ref{REPRES}, $(b)$  by $\nabla^{+}$ and $\nabla^{-}$, respectively. The torsion is given by $T^{\pm}(X, Y)=\pm [X, Y]$.
 Both $\nabla^{\pm}$ can be viewed as special members of   these  bi-invariant linear connections on $G$,  whose Nomizu map induces {\it derivations} on the corresponding Lie algebra $\fr{g}$ (for $\nabla^{-}$ trivially).  In the sequel  we   show  that  this is the desired condition that answers Question \ref{q2}.  
  First we propose  a   formula which allows us to characterize the  $\Ad(G)$-equivariant derivations on $\fr{g}$   in terms   of  the  curvature and the covariant derivative of the  torsion of  a bi-invariant connection satisfying (\ref{stc}). 
   \begin{proposition}\label{DERIV}
  Let $G$ be a compact connected Lie group endowed with  a   bi-invariant affine connection   $\nabla$  whose Nomizu map  $\Lambda : \fr{g}\to\Ed(\fr{g})$ satisfies (\ref{stc}).  Then,    $\Lambda$ is a derivation of $\fr{g}$, i.e.  $\Lambda : \fr{g}\to{\rm Der}(\fr{g})$,   if and only if the curvature $R$ and   the covariant derivative of the torsion $T$  of $\nabla$ satisfy the following relation: 
\begin{equation}\label{stary}
(\nabla_{Z}T)(X, Y)=2\big\{ R(Z, X)Y- \Lambda(Y)\big([Z, X]-\Lambda(Z)X\big)\big\}, \quad \forall  \ X, Y, Z\in\fr{g}.   
\end{equation}
    \end{proposition}
  \begin{proof}
  The proof is direct.  Crucial is our assumption  $\Lambda(X)X=0$ and hence we  mention that for bi-invariant  connections  without this property   our claim fails.    For simplicity set $D(Z, X, Y):=\Lambda(Z)[X, Y]-[\Lambda(Z)X, Y]-[X, \Lambda(Z)Y]$  and  notice that the endomorphism $\Lambda(Z) : \fr{g}\to\fr{g}$ is a derivation  if and only if  $D(Z, X, Y)=0$,  for any $X, Y, Z\in\fr{g}$.  Now, for   any $Z\in\fr{g}$ we view the covariant derivative of  the torsion $T(X, Y)=2\Lambda(X)Y-[X, Y]$ as a bilinear map $\nabla_{Z}T : \fr{g}\times\fr{g}\to\fr{g}$. Then, because  $\Lambda(\Lambda(Z)X)Y=-\Lambda(Y)\Lambda(Z)X$ and $\Lambda([Z, X])Y=-\Lambda(Y)[Z, X]$ for any $X, Y, Z\in\fr{g}$, we obtain that
          \begin{eqnarray*}
  (\nabla_{Z}T)(X, Y)&=&\Lambda(Z)T(X, Y)-T(\Lambda(Z)X, Y)-T(X, \Lambda(Z)Y)\\
  &=&2\Lambda(Z)\Lambda(X)Y-\Lambda(Z)[X, Y]-2\Lambda(\Lambda(Z)X)Y\\
  &&+[\Lambda(Z)X, Y]-2\Lambda(X)\Lambda(Z)Y+[X, \Lambda(Z)Y]\\
  &=& 2\Lambda(Z)\Lambda(X)Y+2\Lambda(Y)\Lambda(Z)X-2\Lambda(X)\Lambda(Z)Y-D(Z, X, Y)\\
  &\overset{(\dag)}{=}&2R(Z, X)Y+2\Lambda([Z, X])Y+2\Lambda(Y)\Lambda(Z)X-D(Z, X, Y)\\
   &=& 2R(Z, X)Y-2\Lambda(Y)\big([Z, X]-\Lambda(Z)X\big)-D(Z, X, Y),
  \end{eqnarray*}
  where in  $(\dag)$ we used  $R(Z, X)Y=\Lambda(Z)\Lambda(X)Y-\Lambda(X)\Lambda(Z)Y-\Lambda([Z, X])Y$.  
    \end{proof}

   \begin{corollary}\label{miso}  Let $G$ be a compact connected Lie group endowed with a  bi-invariant affine connection $\nabla$ whose Nomizu map $\Lambda : \fr{g}\to{\rm Der}(\fr{g})\subset\Ed(\fr{g})$   is a  derivation and satisfies (\ref{stc}).  If $\nabla$  is flat, i.e. $R\equiv0$, or $R(Z, X)Y=\Lambda(Y)\big([Z, X]-\Lambda(Z)X\big)$, then the corresponding torsion $T$  is  $\nabla$-parallel.  
      \end{corollary}
  \begin{proof}
 If $R\equiv 0$, then    equation   (\ref{stary}) reduces to   $(\nabla_{Z}T)(X, Y)=-2\Lambda(Y)\big([Z, X]-\Lambda(Z)X\big)$.  Simultaneously,   Lemma \ref{REPRES} ensures  that    $\Lambda=\Lambda^{\pm}$ and then for the left hand side one gets $-2\Lambda(Y)\big([Z, X]-\Lambda(Z)X)=0$. Hence $\nabla T=0$, i.e. $\nabla^{\pm }T^{\pm }=0$. For example, this is the case if $G$ is  semi-simple, since then any derivation   is  inner (however  notice  that in the compact   case   this  argument fails, see Proposition \ref{John}). Now, if $R(Z, X)Y=\Lambda(Y)\big([Z, X]-\Lambda(Z)X\big)$, then it is immediate from (\ref{stary}) that $\nabla T\equiv 0$.  An alternative way that avoids (\ref{stary}) but includes a few more computations  occurs due to the following observation.  For a bi-invariant affine connection on $G$ satisfying our assumptions, it is not difficult to prove that the equation  $R(Z, X)Y=\Lambda(Y)\big([Z, X]-\Lambda(Z)X\big)$ is equivalent to $[\Lambda(Z), \Lambda(Y)]=\Lambda(\Lambda(Z)Y)$ (as an endomorphism of $\fr{g}$),  or in other words 
$[\Lambda(Z), \Lambda(Y)]X=\Lambda(\Lambda(Z)Y)X=-\Lambda(X)\Lambda(Z)Y$  for any $X, Y, Z\in\fr{g}$. By using  this relation and the properties of $\Lambda$,  a straightforward computation shows   that  $(\nabla_{Z}T)(X, Y)=0$  for any $X, Y, Z\in\fr{g}$.
   \end{proof}
  \noindent  By combining this with Lemma \ref{REPRES} we conclude that
   \begin{corollary}\label{GENcs}
  On a  compact connected Lie group $G$ there exist exactly two    bi-invariant affine connections satisfying (\ref{stc}), which are flat and have parallel torsion.   These are the connections $\nabla^{\pm}$ described in Lemma \ref{REPRES} and they coincide with the $\pm 1$-connections of Cartan-Schouten.
      \end{corollary}
   
    \begin{proposition}\label{John} 
 Let $\fr{g}$ be a reductive Lie algebra. Then,  any derivation $D: \fr{g}\to{\rm Der}(\fr{g})$ is given by
  $D(X)=\phi(Z)\oplus \ad(X_{s})$   for some linear map $\phi : \fr{g}_{0}\to\Ed(\fr{g}_{0})$   in the centre $\fr{g}_{0}$. In particular, $H^{1}(\fr{g}, \fr{g})\cong \Ed(\fr{g}_{0})$.  
 \end{proposition}
 \begin{proof}
  Consider the decomposition of $\fr{g}$ into its centre and  semi-simple part,    i.e. $\fr{g}=\fr{g}_{0}\oplus[\fr{g}, \fr{g}]=\fr{g}_{0}\oplus\fr{g}_{ss}$ and   express any $X\in\fr{g}$  in a unique way by $X=Z+X_{s}$,  where $Z\in\fr{g}_{0}$ and $X_{s}\in\fr{g}_{ss}$. 
  Then,  define  $D: \fr{g}\to{\rm End}(\fr{g})$ by $D(X):=\phi(Z)+\ad(X_{s}) (\equiv\phi(Z)+\ad(X))$   for some linear map $\phi : \fr{g}_{0}\to\Ed(\fr{g}_{0})$.  Obviously, this  is  a derivation of $\fr{g}$ and   in order to prove our claim  it is sufficient  to show  that for any $Z\in\fr{g}_{0}$ and $X_{s}\in\fr{g}_{ss}$   derivations of the form $D_{1}(Z) : \fr{g}_{0}\to\fr{g}_{ss}$ and $D_{2}(X_{s}) : \fr{g}_{ss}\to\fr{g}_{0}$ are necessarily  trivial. This follows easily for $D_{1}$, because  the centre of any Lie algebra is a characteristic ideal, i.e. remains invariant under derivations. Consider  now some $\al, \be\in\fr{g}$ with $[\al, \be]\in\fr{g}_{ss}$ and assume that for any $X_{s}\in\fr{g}_{ss}$ the linear map  $D_{2}(X_{s}) : \fr{g}_{ss}\to\fr{g}_{0}$ is a non-trivial derivation.     Then,  $D_{2}(X_{s})$ acts  on $[\al, \be]$ as an  inner derivation, i.e.  $D_{2}(X_{s})[\al, \be]=\ad_{X_{s}}[\al, \be]=[X_{s},[\al, \be]]$. On the other hand   we have that $D_{2}(X_{s})[\al, \be]=[D_{2}(X_{s})\al, \be]+[\al, D_{2}(X_{s})\be]=0$,   since  $D_{2}(X_{s})\al, D_{2}(X_{s})\be\in\fr{g}_{0}$.  Because $X_{s}, [\al, \be]\in\fr{g}_{ss}$, this gives  a contradiction. In this way we conclude that the spaces ${\rm Der}(\fr{g}_{0}, \fr{g}_{ss})$ and ${\rm Der}(\fr{g}_{ss}, \fr{g}_{0})$ must be trivial  and for the Lie algebra ${\rm Der}(\fr{g})$  we get the direct sum decomposition 
  \[
  {\rm Der}(\fr{g})={\rm Der}(\fr{g}_{0})\oplus{\rm Der}(\fr{g}_{ss})=\Ed(\fr{g}_{0})\oplus \ad(\fr{g})={\rm Out}(\fr{g})\oplus{\rm Inn}(\fr{g}),
  \]
   where ${\rm Inn}(\fr{g})\cong \fr{g}\backslash\fr{g}_{0}:=\{\ad(X) : X\in\fr{g}\}=\ad(\fr{g})$ denotes the space of all inner derivations  (the adjoint algebra) and ${\rm Out}(\fr{g})$ is the quotient algebra of outer derivations, i.e. ${\rm  Out}(\fr{g})\cong {\rm Der}(\fr{g})\backslash{\rm  Inn}(\fr{g})$. 
  The algebra  ${\rm Out}(\fr{g})$ coincides with the first cohomology $H^{1}(\fr{g}, \fr{g})$   of $\fr{g}$ acting on itself by the adjoint representation, see \cite[p.~57]{GOV}.  Hence  $H^{1}(\fr{g}, \fr{g})\cong \Ed(\fr{g}_{0})$.
              \end{proof}

\noindent  Therefore, given a compact connected Lie group $G$ and an arbitrary   derivation $D : \fr{g}\to{\rm Der}(\fr{g})$, the relation $D(X)X=0$   is {\it not} necessarily true for any $X\in\fr{g}$.  Next we will show  that if  $D : \fr{g}\to{\rm Der}(\fr{g})$ is  an $\Ad(G)$-equivariant derivation with  $D(X)X=0$ for any $X\in\fr{g}$, then $\phi$   must be   trivial $\phi\equiv 0$, i.e. $D$ is an inner derivation. Although for non-central elements $g\notin Z(G)$ one can prove this result  easily, for central elements the equivariance condition does not provide any further information  and a proof of the claim seems difficult.  Proposition \ref{DERIV}  allows us to overpass this problem. In fact, we  provide two different proofs with the first one being  independent of Laquer's  classification results \cite{Laq1}.   
        \begin{proposition}\label{CLG}
   Let $G$ be a  compact connected Lie group  and let $D : \fr{g}\to{\rm Der}(\fr{g})$ be a derivation of $\fr{g}=T_{e}G$. Assume that   $D(\Ad(g)X)=\Ad(g)D(X)\Ad(g)^{-1}$ for any $g\in G$, $X\in\fr{g}$ and   that   $D(X)X=0$  for any $X\in\fr{g}$.   Then  $D$ is an inner derivation.    \end{proposition} 
   \begin{proof}
     {\bf 1st way:}  By Proposition \ref{John}, write  $D=\phi\oplus\ad$ for some linear map $\phi : \fr{g}_{0}\to\Ed(\fr{g}_{0})$.  Because   $\Ad(G)\fr{g}_{0}=\fr{g}_{0}$ and the adjoint representation of $\fr{g}$ is $\Ad(G)$-equivariant, it turns out that $D$ has  the same property if and only if    $\phi(\Ad(g)Z)=\phi(Z)(\equiv \Ad(g)\phi(Z)\Ad(g)^{-1})$, for any     $g\in G$ and $Z\in\fr{g}_{0}$  (where we view $\Ad(g)\phi(Z)\Ad(g)^{-1}$   as an  endomorphism $\fr{g}_{0}\to\fr{g}_{0}$). In addition, the condition $D(X)X=0$ for any $X\in\fr{g}$  is equivalent to   $\phi(Z)Z=0$ for any $Z\in\fr{g}_{0}$.
Now, it is sufficient to prove that $\phi\equiv 0$. Assume in contrast that $\phi(Y)\neq 0$ for some $Y\in\fr{g}_{0}$.    We view the centre $Z(G)$ of $G$ as a compact Lie group itself  and we identify $T_{e}(Z(G))=\fr{g}_{0}$ (the centre   $Z(G)$ is closed subgroup of $G$).  Because for any $Z\in\fr{g}_{0}$ the endomorphism $\phi(Z) : \fr{g}_{0}\to\fr{g}_{0}$ is (trivially) a derivation  which satisfies the properties of  Proposition \ref{DERIV},   the associated bi-invariant affine connection on $Z(G)$   satisfies (\ref{stary}) for any $X, Y, Z\in\fr{g}_{0}$.  Let us denote this connection by $\nabla^{\phi}$.  Obviously $R^{\phi}\equiv 0$ (since $\fr{g}_{0}$ is abelian) and $T^{\phi}(Z, Z')=\phi(Z)Z'-\phi(Z')Z$ for any $Z, Z'\in\fr{g}_{0}$. An easy computation also shows that   
   \begin{eqnarray*}
   (\nabla^{\phi}_{Z}T^{\phi})(X, Y)&=&\phi(Z)(T^{\phi}(X, Y))-T^{\phi}(\phi(Z)X, Y)-T^{\phi}(X, \phi(Z)Y)\\
  &=&\phi(\phi(Z)Y)X-\phi(\phi(Z)X)Y,
  \end{eqnarray*}
     since for example $\phi(Z)\phi(X)=\phi(X)\phi(Z)$ for any $Z, X\in\fr{g}_{0}$. Finally,  relation (\ref{stary})  becomes
     \[
      \phi(\phi(Z)Y)X-\phi(\phi(Z)X)Y=2\phi(Y)\phi(Z)X \quad \ \forall \ X, Y, Z\in\fr{g}_{0}. 
     \]
Now, for $X=Z$ it reduces to $\phi(\phi(X)Y)X=0$, for any $X, Y\in\fr{g}_{0}$. Because the identity $\phi(Z)Z=0$ is equivalent to $\phi(X)Y+\phi(Y)X=0$ for any $X, Y\in\fr{g}_{0}$, we can   write $\phi(\phi(Y)X)X=0$. Hence,  it must be $\phi(Y)X=X$, i.e. $\phi(Y) : \fr{g}_{0}\to\fr{g}_{0}$ is the identity map for any $Y\in\fr{g}_{0}$. But then $\phi(Y)Y=Y$,  i.e. $\phi(Y)Y\neq 0$ for any $Y\in\fr{g}_{0}$, which gives rise to a contradiction. Thus $\phi\equiv 0$ and $D\equiv\ad$, as claimed.
    
    \smallskip
       \noindent{\bf 2nd way:}    Proposition \ref{DERIV} characterizes the $\Ad(G)$-equivariant derivations  $D : \fr{g}\to{\rm Der}(\fr{g})$ on the Lie algebra $\fr{g}$ of a   compact connected  Lie group $G$ satisfying  the condition $D(X)X=0$  for any $X\in\fr{g}$.  Such derivations correspond to    bi-invariant affine  connections of $G$, whose Nomizu map $\Lambda : \fr{g}\to\fr{gl}(\fr{g})$    satisfies   the relations (\ref{stc}) and  (\ref{stary}).   For a  compact    simple  Lie group,  except $G=\SU(n)$ $(n\geq 3)$, the   bi-invariant affine connections   are described (up to   scale) by the 1-parameter family  $\Lambda^{\al} : \fr{g}\to\Ed(\fr{g})$  with $\Lambda^{\al}:=((1-\al)/2) \cdot \ad$, for some $\al\in\bb{R}$, which is obviously an  $\Ad(G)$-equivariant  derivation.  For the general compact case, by  adopting the notation of Theorem \ref{AFHoll}   it is easy to see that  for some $\al_{i}\in\bb{R}$  the expression $D(X):=\sum_{i=1}^{r}\Lambda^{\al_{i}}(X)=\sum_{i=1}^{r}((1-\al_{i})/2)\cdot \ad(X)|_{\fr{g}_{i}}$  is a  derivation on $\fr{g}$, which   turns out to be inner (by linearity of $\ad$).  In order to prove our claim, there remains to exclude the exotic connections of $\SU(n)$ and $\U(n)$.          Indeed,  a routine  computation shows that these are not derivations,  in particular: {\it the unique  bi-invariant linear connections  of a compact connected Lie group which induce derivations on the corresponding  Lie algebra are induced by the Lie bracket}.   For example, for $\SU(n)$  the  bilinear map     $\eta^{\rm exc}$ described in Theorem \ref{AFHoll}  does not induce a derivation    nor   does it  satisfy  (\ref{stc}).   The special families  for  $\U(n)$ are more complicated.  For $n=2$, aside the skew-symmetric map induced by the Lie bracket, the new families of symmetric bilinear $\Ad(\U(n))$-equivariant maps   span a 3-dimensional space \cite{Laq1}.  However, neither these are derivations    and the condition  $\eta(X, X)=0$ also fails. The same is true for  $n\geq 3$;  there is a 3-dimensional space generated by symmetric bilinear maps $\eta_{i} : \fr{u}(n)\times\fr{u}(n)\to\fr{u}(n)$ which do not  induce derivations, namely $\eta_{1}(X, Y):= i (X\cdot Y+Y\cdot X)$, $\eta_{2}(X, Y):={\rm tr}(X\cdot Y)\cdot iI$, and $\eta_{3}(X, Y):={\rm tr}(X){\rm tr}(Y)\cdot iI$, but also the skew-symmetric   map $\mu(X, Y)=i\big({\rm tr}(Y)X-{\rm tr}(X)Y\big)$  (see \cite[Thm.~10.1]{Laq1} or \cite[Thm.~3.1]{AFH}). Because   $\mu(X, X)=0$  for any $X\in\fr{u}(n)$,   $\mu$ is at least a  candidate of Proposition \ref{DERIV}.   However  a quick check implies that neither this is a derivation.   Now, although a linear combination $\eta_{c}(X, Y):=c_{1}\eta_{1}(X, Y)+c_{2}\eta_{2}(X, Y)+c_{3}\eta_{3}(X, Y)+c\mu(X, Y)$  gives rise to  an $\Ad(\U(n))$-equivariant bilinear map on $\fr{g}$,   the condition  $\eta_{c}(X, X)=0$ for any $X\in\fr{u}(n)$ is true, if and  only if, $c_{1}=c_{2}=c_{3}=0$.
   Because  on an arbitrary compact Lie group $G$ these connections exhaust all possible   bi-invariant affine connections \cite[Thm. 9.1]{Laq1},  the proof is complete.
   \end{proof}
      Based on Proposition  \ref{CLG}, we are now able  to present  the main  theorem of this section. 
                 \begin{theorem} \label{THM1}  
    Let $G$ be  a compact connected Lie group with a bi-invariant metric $\rho$  and let $\nabla$ be a bi-invariant  affine connection corresponding to a linear $\Ad(G)$-equivariant map    $\Lambda : \fr{g}\to\fr{gl}(\fr{g})$ satisfying (\ref{stc}).   Then $\Lambda$  is a derivation  if and only if  $\nabla$ is metric with respect to $\rho$.   In particular, the class of   bi-invariant {\it affine}  connections  which induce derivations $\Lambda  : \fr{g}\to{\rm Der}(\fr{g})$ on the corresponding Lie algebra $\fr{g}$  coincides with the class of  bi-invariant  {\it metric} connections on $G$.
     \end{theorem}
         \begin{proof}
For the first part  we    need only to prove   the one direction, since  the converse is obvious due to Theorem \ref{AFHoll}.  Recall  that the Killing form of a Lie algebra $\fr{g}$  (which here we denote by $B\equiv B_{\fr{g}}$)  satisfies the relation $B(AX, AY)=B(X, Y)$,  for any  automorphism $A : \fr{g}\to\fr{g}$, see \cite[p.~13]{GOV}.   If $\Lambda(X)\in{\rm Der}(\fr{g})$, then ${\rm exp}(t\Lambda(X))\in{\rm Aut}(\fr{g})$.   Thus, the derivative of the relation $B(AX, AY)=B(X, Y)$ at $t=0$  for  $A={\rm exp}(t\Lambda(X))$  implies that $B(\Lambda(X)Y, Z)+B(Y, \Lambda(X)Z)=0$ for any $X, Y, Z\in\fr{g}$.  If the Lie group $G$ is simple, then   any $\Ad(G)$-invariant inner product is a multiple  of   $-B$, so $\Lambda(X)\in\fr{so}(\fr{g})$.  If $G$ is just  compact, then  we express the  $\Ad(G)$-invariant inner product associated to $\rho$ by  $\langle \ , \ \rangle=  b|_{\fr{g}_{0}}-\sum_{i=1}^{r}c_{i}\cdot B|_{\fr{g}_{i}}$  for some  $c_{i}>0$ (see Theorem \ref{AFHoll}). This  is possible, because $\rho|_{\fr{g}_{i}}=\text{multiple of} \ -B_{i}$ where     $B_{i}\equiv B_{\fr{g}_{i}}=B|_{\fr{g}_{i}}$, for any $i=1, \dots, r$.  
 As we   explained above,  for  any simple ideal $\fr{g}_{i}$  $(1\leq i\leq r)$ (inner) derivations   become   metric  with respect to $B|_{\fr{g}_{i}}$.  For the centre $\fr{g}_{0}$ not all the derivations are necessarily  metric with respect to the scalar product $b$.  However,    $\Lambda : \fr{g}\to{\rm Der}(\fr{g})$ is  an $\Ad(G)$-equivariant derivation with $\Lambda(X)X=0$ and Proposition  \ref{CLG} guarantees that this is inner. Hence the centre has no contribution and  we finally obtain $\Lambda(X)\in\fr{so}(\fr{g})$  for any $X\in\fr{g}$.    Now, the final result is valid   if one can drop  the condition $\Lambda(X)X=0$. Indeed, this is the case because   a routine computation shows that a linear combination of the exotic connections on $\U(n)$  $(n\geq 2)$ fails to  induce   a non-trivial derivation (nor satisfies (\ref{stc}) as we explained in  the proof of  Proposition \ref{CLG}).   The same time, the endomorphsim $\Lambda(X):=\sum_{i=1}^{r}((1-\al_{i})/2)\cdot \ad(X)|_{\fr{g}_{i}}$  is an equivariant  derivation which trivially verifies the condition $\Lambda(X)X=0$.
  \end{proof}

\section{Invariant metric connections with skew-torsion on naturally reductive spaces}
 Let $(M=G/K, g)$ be a connected naturally reductive Riemannian manifold.  We shall  study  $G$-invariant metric connections whose torsion is proportional to the torsion   of the canonical connection $\nabla^{c}$.  Let  $\fr{g}=\fr{k}\oplus\fr{m}$ be  a reductive decomposition and let $\Lambda^{g} :\fr{m}\to\fr{so}(\fr{m})$  be the Nomizu map of the Levi-Civita  connection. 
   \begin{proposition}\label{Ilkas} 
     $(a)$  For any $\al\in\bb{R}$ there is a bijective correspondence between   linear  $\Ad(K)$-equivariant maps  $\Lambda^{\al}: \fr{m}\to\fr{so}(\fr{m})$, defined by
     \begin{equation}\label{nal}
    \Lambda^{\al}(X)Y=\frac{1-\al}{2}[X, Y]_{\fr{m}}=(1-\al)\Lambda^{g}(X)Y, \quad \forall \ X, Y\in\fr{m},
     \end{equation}
       and   $G$-invariant metric connections $\nabla^{\al}$ on $T(G/K)$ with  skew-symmetric torsion $T^{\al}\in\Lambda^{3}(\fr{m})$ such that $T^{\al}=\al \cdot T^{c}$.
          
          \noindent $(b)$ If    the   Lie group $G$ is  compact and simple, then   the family  $\{\nabla^{\al} : \al\in\bb{R}\}$   is naturally induced    by a  bi-invariant   metric connection of   $G$.
               \end{proposition}
                    \begin{proof}    
  $(a)$ The direct statement is   well-known \cite{Agr03}. 
  The converse  is also very easy.  Because $(M=G/K, g)$ is naturally reductive with respect to $G$,     $\nabla^{\al}$ is a $G$-invariant metric connection with skew-torsion $T^{\al}\in\Lambda^{3}(\fr{m})$  if and only if   the corresponding Nomizu map, say  $\Lambda_{\fr{m}} : \fr{m}\to\fr{so}(\fr{m})$, is such that $\Lambda_{\fr{m}}(X)Y+\Lambda_{\fr{m}}(Y)X=0$ for any $X, Y\in\fr{m}$, see  \cite[Lem.~2.1]{Agr03}.    Because $T^{\al}=\al \cdot T^{c}$, a simple application of (\ref{torsion}) shows that   $\Lambda_{\fr{m}}(X)Y-\Lambda_{\fr{m}}(Y)X=(1-\al)[X, Y]_{\fr{m}}$, for any $X, Y\in\fr{m}$  and   the claim follows.   
       
 \noindent  $(b)$   
  In  \cite[Thm.~6.1]{Laq2} it is proved that  there is  a natural mapping 
\[
 \fr{aff}_{G\times G}(F(G))\to \fr{aff}_{G}(F(G/K)), \quad \eta\mapsto \pi_{*}\eta, \ \text{with} \  (\pi_{*}\eta)(X, Y):=\eta(X, Y)_{\fr{m}},
 \]
for any $X, Y\in\fr{m}$. 
Here, $\eta$ is a bi-invariant linear connection on $G$ and  $\pi_{*}\equiv d\pi_{e} : \fr{g}\to\fr{m}$  is the differential of $\pi$ at $e$.   If $G$ is compact and $\eta$ is a bi-invariant {\it metric} connection on $G$, then  the  induced $G$-invariant connection on $M=G/K$ will be also metric,  since the inner  product on $\fr{m}$ is the restriction of an $\Ad(G)$-invariant inner product of $\fr{g}$. In our case, and since $g$ has been assumed to be naturally reductive,  $\eta$ oughts to induce    a  $G$-invariant metric connection with  skew-torsion.
 For $G$ compact and simple, any bi-invariant metric connection is given by the map $\eta^{\al}(X, Y)=((1-\al)/2)[X, Y]$ for some  $\al\in\bb{R}$    (up to scale).  Consider  the composition $\pi_{*}\eta^{\al} : \fr{g}\times\fr{g}\to\fr{m}$ and write  $X=X_{\fr{m}}+X_{\fr{k}}$. Then, by restricting $\pi_{*}\eta^{\al}$ on $\fr{m}\times\fr{m}$ we obtain a well-defined bilinear map   $\lambda^{\al} : \fr{m}\times\fr{m}\to\fr{m}$   with 
$
\lambda^{\al}(X_{\fr{m}}, Y_{\fr{m}})=(\pi_{*}\eta^{\al})(X_{\fr{m}}, Y_{\fr{m}}):=\eta^{\al}(X_{\fr{m}}, Y_{\fr{m}})_{\fr{m}}=((1-\al)/{2})[X_{\fr{m}}, Y_{\fr{m}}]_{\fr{m}}.
$
 This is an $\Ad(K)$-equivariant map  satisfying  $\langle \lambda^{\al}(X_{\fr{m}}, Y_{\fr{m}}), Z_{\fr{m}}  \ \rangle+\langle Y_{\fr{m}}, \lambda^{\al}(X_{\fr{m}}, Z_{\fr{m}})\rangle=0$.      The associated Nomizu map $\Lambda^{\al} : \fr{m}\to\fr{so}(\fr{m})$  is  the   family   discussed in $(a)$.   \end{proof}

   We recall  now the case of a   symmetric space of Type I.   
    \begin{theorem}\textnormal{(\cite{Laq2, AFH})}\label{symspace}
 Let $(M=G/K, g)$ be  an (irreducible)   Riemannian symmetric space of Type I.  Then
 
 \noindent (a)
 If $\nabla$ is a $G$-invariant metric connection with torsion a multiple of the torsion of the canonical connection, then necessarily $\nabla\equiv \nabla^{c} (\equiv\nabla^{g})$. 
 
\noindent  (b) The space of $G$-invariant metric connections  consists   of  just the canonical connection $\nabla^{c}\equiv\nabla^{g}$.

  \end{theorem}

  \begin{remark}\label{nonsymspaces}
  \textnormal{According to \cite[Thm.~1.2]{Olmos}, given a compact naturally reductive Riemannian manifold $(M=G/K, g)$ (locally irreducible)    the canonical connection is   {\it unique}   under the assumption that $M$ is not isometric to a sphere,    a real projective space, or a compact  simple  Lie group with a bi-invariant metric.    Viewing the sphere  $\Ss^{n}$ as a compact quotient  $M=G/K$,  this anomaly appears since  $G$ is not necessarily equal to the full isometry group ${\rm Iso}(M)$ or its connected component ${\rm Iso}_{0}(M)$, in contrast to a  symmetric space of Type I.   Actually, let $(M=G/K, g_{B})$ be an (effective) simply connected normal homogeneous   manifold with $G$ being a compact,  connected, simple  Lie group and assume that   the isotropy representation is (strongly) irreducible. Then   $G={\rm Iso}_{0}(M)$, unless $M=\G_2/\SU(3)=\Ss^{6}$ or $M=\Spin(7)/\G_2=\Ss^{7}$ where ${\rm Iso}_{0}(M, g_{B})=\SO(7), \SO(8)$, respectively (see    \cite[Thm.~17.1]{Wolf} or \cite[p.~623]{Wa1}). It is well-known that  there are more spheres that can be represented  as quotients of   distinct Lie groups \cite{Bes}. In particular, the theory of enlargements of transitive actions (developed by A. L. Oni\v{s}\v{c}ik)  describes   all simple compact Lie algebras $\fr{g}$ which can be written as a direct sum $\fr{g}=\fr{k}_{1}\oplus\fr{k}_{2}$ of two Lie subalgebras $\fr{k}_{1}, \fr{k}_{2}$ (see  \cite{Ke, KS}).  If $G$ is the compact  simply connected Lie group corresponding to $\fr{g}$ and $K_{1}, K_{2}\subset G$ are the Lie subgroups associated to   $\fr{k}_{1}, \fr{k}_{2}$, then it holds that $\fr{g}=\fr{k}_{1}\oplus\fr{k}_{2}$   if and only if  $K_{1}$ acts transitively on $G/K_{2}$.  Hence, in the Lie group level we have the identifications $G/K_{1}=K_{2}/(K_{1}\cap K_{2})$ (and $G/K_{2}=K_{1}/(K_{1}\cap K_{2})$).     Oni\v{s}\v{c}ik's list (for symmetric cosets)  contains several spheres.    Let us present them.
  \[
  \begin{tabular}{l| l| l| l| l}
 & $G/K_{1}$ & $\fr{p}$ &  $K_{2}/K_{1}\cap K_{2}$ & $\fr{m}$ \\
   \hline
  $\Ss^{4n-1}$ & $\SO(4n)/\SO(4n-1)$  &   irred. & $\Sp(n)/\Sp(n-1)$ & $\fr{m}_{1}\oplus\fr{m}_{2}$ \\
  $\Ss^{4n-1}$ & $\SO(4n)/\SO(4n-1)$ &  irred. &  $\Sp(n)\U(1)/\Sp(n-1)\U(1)$ & $\fr{m}_{1}\oplus\fr{m}_{2}$  \\
    $\Ss^{4n-1}$ & $\SO(4n)/\SO(4n-1)$ & irred. &  $\Sp(n)\Sp(1)/\Sp(n-1)\Sp(1)$ & $\fr{m}_{1}\oplus\fr{m}_{2}\oplus\fr{m}_{3}$ \\
  $\Ss^{2n-1}$ & $\SO(2n)/\SO(2n-1)$ &  irred. & $\U(n)/\U(n-1)$ & $\fr{m}_{1}\oplus\fr{m}_{2}$ \\
  $\Ss^{2n-1}$ & $\SO(2n)/\SO(2n-1)$ &  irred. & $\SU(n)/\SU(n-1)$ & $\fr{m}_{1}\oplus\fr{m}_{2}$ \\
  $\Ss^{6}$ & $\SO(7)/\SO(6)$ &  irred. & $\G_2/\SU(3)$ & irred.   \\
   $\Ss^{7}$ & $\SO(8)/\SO(7)$ &  irred. &  $\Spin(7)/\G_2$ & irred.   \\
    $\Ss^{15}$ & $\SO(16)/\SO(15)$ &  irred. &  $\Spin(9)/\Spin(7)$ &  $\fr{m}_{1}\oplus\fr{m}_{2}$ 
        \end{tabular}
  \]
 In this table,  any symmetric space $M=G/K_1$ is   isotropy irreducible, but for the presentations $K_{2}/(K_{1}\cap K_{2})$ only them  of $\Ss^{6}$ and  $\Ss^{7}$ are (strongly)  isotropy irreducible. Another fact that deserves our attention is that although the cosets $K_{2}/(K_{1}\cap K_{2})$ are diffeomorphic to a Riemannian symmetric space, namely a sphere,  the  pairs $(K_{2}, K_{1}\cap K_{2})$ are not necessarily symmetric.  For example 
   \[
  \Ss^{7}=\SO(8)/\SO(7)\cong \U(4)/\U(3) \cong   \Sp(2)\U(1)/\Sp(1)\U(1)\cong  \Spin(7)/\G_2,
  \] 
 \noindent   but taking a reductive decomposition $\fr{g}=\fr{k}\oplus\fr{p}$   the relation $[\fr{p}, \fr{p}]\subset\fr{k}$  holds only for the first presentation.   
   Due to this observation and since  $K_{2}\subset G$ and $M=G/K_{1}=K_{2}/(K_{1}\cap K_{2})$, one may expect  more $K_{2}$-invariant affine connections on $M$ than $G$-invariant connections. Let us examine this interesting  problem   for the irreducible cosets $K_{2}/(K_{1}\cap K_{2})$ appearing   above. }
    \end{remark}
  \begin{theorem}\label{THM2}
    The space of $\Spin(7)$-invariant affine (or metric) connections on the 7-sphere $\Ss^{7}=\Spin(7)/\G_2$ is 1-dimensional; it consists of the family 
     $\{\nabla^{\al} : \al\in\bb{R}\}$ described in Proposition \ref{Ilkas}, i.e.   $\dim_{\bb{R}}(\fr{aff}_{\Spin_{7}}(F(\Ss^{7}))=1$.
    Similarly, the space of  \ $\G_2$-invariant affine (or metric) connections on the sphere $\Ss^{6}=\G_2/\SU(3)$ consists of the same family $\nabla^{\al}$, but in this case the parameter $\al$ is a complex number  $\al\in\bb{C}$.  Hence $\dim_{\bb{R}}(\fr{aff}_{\G_{2}}(F(\Ss^{6}))=2$. 
            \end{theorem}
\begin{proof}
 The non-symmetric presentations of $\Ss^{6}$ and $\Ss^{7}$ are still (strongly) isotropy irreducible. Therefore, in order to compute  the dimensions of the spaces   $\fr{aff}_{\G_{2}}\big(F(\G_{2}/\SU(3))\big)$ and $\fr{aff}_{\Spin(7)}\big(F(\Spin(7)/
  \G_{2})\big)$, it is sufficient to find the multiplicity of the corresponding isotropy representation  $\fr{m}$ inside $\fr{m}\otimes\fr{m}=\Lambda^{2}(\fr{m})\oplus \Sym^{2}(\fr{m})$. However,  we need now to view $\fr{m}$ as a $\SU(3)$- (resp. $\G_2$-) module.      
      Consider first the  7-sphere $\Ss^{7}\subset \bb{R}^{8}$ and  identify $\bb{R}^{8}\cong \bb{O}$, where $\bb{O}$ are  the Cayley numbers.  We view $\G_2\cong\Aut(\bb{O})$ as a subgroup  of $\Spin(7)\subset\Cl(\bb{R}^{7})$ preserving the spinor $\psi_{0}=(1, 0, \ldots, 0)^{t}\in\Delta_{7}$, where $\Delta_{7}:=\bb{R}^{8}$ is the 8-dimensional  spin representation of $\Spin(7)$.  Because $\Spin(7)$ acts transitively on $\Ss^{7}$ we get the diffeomorphism $\Ss^{7}\cong \Spin(7)/\G_2$, see  \cite{FKMS}.   As usual,    we  write  $V^{a, b}$ for the  irreducible representation of $\G_2$  corresponding to highest weight  $(a, b)$, where both $a, b$ are non-negative integers; for example $V^{0, 0}\cong\bb{R}$ is the trivial representation, $\phi_{7}:=V^{1, 0}\cong\bb{R}^{7}$  is the standard representation  of $\G_2$ and $V^{0, 1}\cong\fr{g}_{2}$ is its adjoint representation. 
         Let  now $\fr{spin}(7)=\fr{g}_{2}\oplus\fr{m}$ be a reductive decomposition.  The isotropy representation  $\fr{m}$ coincides with the standard representation $\fr{m}\cong\phi_{7}=\{X \lrcorner \omega : X\in\bb{R}^{7}\}\cong\bb{R}^{7}$, where  $\omega$ states for the (generic) 3-form on $\bb{R}^{7}$ preserved by $\G_2$, see  \cite{FKMS, FrIv}. For the (real) $\G_2$-modules $\Lambda^{2}(\fr{m})$ and $\Sym^{2}(\fr{m})$ we get the decompositions (we use the Lie software package, for $\Lambda^{2}(\fr{m})$ see also \cite{FKMS, FrIv}):
  \[
  \Lambda^{2}(\fr{m}) \cong\fr{so}(7)= V^{0, 1}\oplus V^{1, 0}= \fr{g}_{2}\oplus\phi_{7}=\fr{g}_{2}\oplus\fr{m},  \quad \Sym^{2}(\fr{m})=V^{2, 0}\oplus \bb{R},
  \]
  where $V^{2, 0}\cong \Ss^{2}_{0}\bb{R}^{7}$  with $\dim V^{2, 0}=27$. Thus, there is only one copy of $\fr{m}$ inside the $\G_2$-module $\fr{m}\otimes\fr{m}$, lying in $\Lambda^{2}(\fr{m})$. In other words, there is skew-symmetric bilinear $\Ad(\G_2)$-equivariant map $\eta : \fr{m}\times\fr{m}\to\fr{m}$ which induces a 1-dimensional  family of $\Spin(7)$-invariant affine connections on $\Ss^{7}$. Because $\fr{m}$ is irreducible,  Schur's lemma tells us that $\eta$ must be a multiple of the Lie bracket, say $\eta(X, Y)=\frac{(1-\al)}{2}\cdot [X, Y]_{\fr{m}}$ for some $\al\in\bb{R}$, with  $X, Y\in\fr{m}$. This defines the family $\{\nabla^{\al} : \al\in\bb{R}\}$ discussed in  Proposition \ref{Ilkas}.
      We  treat now  the 6-sphere. Recall  that $\G_2$ preserves the imaginary octonions  ${\rm Im}(\bb{O})\cong\bb{R}^{7}$ and acts transitively on  $\Ss^{6}\subset {\rm Im}(\bb{O})$ with stabilizer   diffeomorphic to $\SU(3)$, i.e. $\Ss^{6}\cong \G_2/\SU(3)$, see  \cite[Lemma.~5.1]{Ke}.    The weights of $\SU(3)$  are also given by pairs of non-negative  integers  $(a, b)$ and  irreducible  $\SU(3)$-representations will be labeled again by $V^{a, b}$.  In particular,   it is $\dim_{\bb{C}} V^{a, b}=\frac{1}{2}(a+1)(b+1)(a+b+2)$. 
    Obviously, $V^{1, 0}\cong\bb{C}^{3}:=\mu_{3}$ is the standard (complex) representation of $\SU(3)$, $V^{0, 1}\cong\overline{\bb{C}^{3}}=\overline{\mu_{3}}$ is its conjugate and $V^{1, 1}\cong\fr{su}(3)^{\bb{C}}$ is the  complexified  adjoint representation. 
   Let  $\fr{g}_{2}=\fr{su}(3)\oplus\fr{m}$ be a reductive decomposition.  It follows that $\fr{m}=[\mu_{3}]_{\bb{R}}$, where for a complex representation $V$ we denote by $[V]_{\bb{R}}$ the underlying real representation (whose real dimension is twice the complex dimension of $V$).  Thus, it  is more convenient to use  the complexified isotropy representation, which  splits  into two conjugate (inequivalent) submodules: $\fr{m}\otimes_{\bb{R}}\bb{C}=\mu_{3}\oplus\overline{\mu}_{3}=\bb{C}^{3}\oplus\overline{\bb{C}^{3}}$. Then, for the $\SU(3)$-module $\Lambda^{2}(\fr{m})\otimes_{\bb{R}}\bb{C}$ we get
 \begin{eqnarray*}
  \Lambda^{2}(\fr{m})\otimes_{\bb{R}}\bb{C}&=& \Lambda^{2}(\fr{m}^{\bb{C}})=\Lambda^{2}(\bb{C}^{3}\oplus\overline{\bb{C}^{3}})=\Lambda^{2}(\bb{C}^{3})\oplus\Lambda^{2}(\overline{\bb{C}^{3}})\oplus (\bb{C}^{3}\otimes \overline{\bb{C}^{3}})=(V^{1, 0}\oplus V^{0, 1})\oplus V^{1, 1}\oplus \bb{C},
  \end{eqnarray*}
  since $\bb{C}\cong V^{0, 0}$, $\Lambda^{2}(\bb{C}^{3})\cong\Lambda^{2}(\mu_{3})\cong V^{1, 0}=\mu_{3}$ (see also \cite[p.~125]{Ke2}) and $\Lambda^{2}(\overline{\bb{C}^{3}})\cong \overline{\Lambda^{2}(\bb{C}^{3})}\cong \overline{V^{1, 0}}\cong V^{0, 1}=\overline{\mu_{3}}$.   Hence we finally conclude that 
    \[
  \Lambda^{2}(\fr{m})=[V^{1, 0}]_{\bb{R}}\oplus\fr{su}(3)\oplus\bb{R}=[\mu_{3}]_{\bb{R}}\oplus\fr{su}(3)\oplus\bb{R}=\fr{m}\oplus\fr{su}(3)\oplus\bb{R}.
  \]
Under the action of $\SU(3)$,   we also compute
  \begin{eqnarray*}
   \Sym^{2}(\fr{m})\otimes_{\bb{R}}\bb{C}&=&\Sym^{2}(\fr{m}^{\bb{C}})=\Sym^{2}(\bb{C}^{3}\oplus\overline{\bb{C}^{3}})=\Sym^{2}(\bb{C}^{3})\oplus \Sym^{2}(\overline{\bb{C}^{3}})\oplus (\bb{C}^{3}\otimes \overline{\bb{C}^{3}})\\
   &=&(V^{2, 0}\oplus V^{0, 2})\oplus V^{1, 1}\oplus\bb{C},
   \end{eqnarray*}   
      with $V^{2, 0}\cong\Sym^{2}(\bb{C}^{3})$.  
      Consequnetly,   the decomposition of    $\Sym^{2}(\fr{m})$ into irreducible $\SU(3)$-submodules  is given by   $\Sym^{2}(\fr{m})=[V^{2, 0}]_{\bb{R}}\oplus\fr{su}(3)\oplus\bb{R}$.
Similarly with $\Ss^{7}$,  the copy of $\fr{m}$ inside $\Lambda^{2}(\fr{m})$ defines a skew-symmetric bilinear $\Ad(\SU(3))$-equivariant map $\eta : \fr{m}\times\fr{m}\to\fr{m}$, which  by Schur's lemma, must be  proportional to the Lie bracket restricted on $\fr{m}$.  In fact, in this case  the parameter $\al$ is a complex number, i.e.  $\eta(X, Y)=\frac{(1-\al)}{2}\cdot [X, Y]_{\fr{m}}$ for some $\al\in\bb{C}$. This  
proves the claim for the affine case. Now, the assertion about the metric property is simple.  Any $\Spin(7)$-invariant metric on  $\Ss^{7}=\Spin(7)/\G_2$ must be  a multiple of the negative of the Killing form of $\Spin(7)$, restricted on $\fr{m}$. Hence, the family $\{\nabla^{\al} : \al\in\bb{R}\}$ is necessarily metric.  Similarly for  $\Ss^{6}=\G_2/\SU(3)$. 
      \end{proof}
        \begin{remark}\label{notinv}
  \textnormal{The embedding of $\Ss^{7}$ inside the spin representation $\Delta_{7}\cong\bb{R}^{8}$ induces on $\Ss^{7}$  an affine  connection $\nabla^{\rm flat}$ which is  metric and has (non-parallel) skew-torsion $T^{\rm flat}\neq 0$ \cite{FLAT}.  The 7-sphere   endowed with this connection and a Riemannian metric of constant sectional curvature becomes   flat, and   together with  the compact  (simple)  Lie groups  endowed with a bi-invariant metric and  one of the $\pm 1$-connections,  exhaust all     Riemannian manifolds carrying a flat metric connection  with non-trivial skew-torsion (Cartan-Schouten theorem).   Viewing   the sphere $\Ss^{7}=\Spin(7)/\G_2$ as a $\G_2$-manifold, I.~Agricola and Th.~Friedrich \cite[pp.~7--9]{FLAT} described this connection  as a connection    whose torsion $T^{\rm flat}$   does not have {\it constant} coefficients. Hence, $\nabla^{\rm flat}$ is {\it not} an invariant connection and this is the reason that it does not appear in Theorem \ref{THM2} (for example, the difference $\nabla^{\rm flat}-\nabla^{c}$ is not an $\Ad(\G_2)$-invariant tensor and hence given a  reductive decomposition $\fr{spin}(7)=\fr{g}_{2}\oplus\fr{m}$, 
  the relation    $T^{\rm flat}=\al\cdot T^{c}$ fails for any $\al$).\footnote{The author thanks I. Agricola for this remark.} }
    \end{remark}
      \section{$\nabla^{\al}$-Einstein naturally reductive manifolds with skew-torsion}\label{CST}
In this section we describe the geometry of  a naturally reductive manifold  endowed with a family of invariant metric connections whose torsion is such that $T^{\al}=\al\cdot T^{c}$, for some parameter $\al$. We are mainly  interested to answer the following question: {\it For which values of $\al$  $(\al\neq 0)$  the induced  Ricci tensor $\Ric^{\al}$ is ``proportional'' to the naturally reductive metric $g$?} 

  \subsection{$\nabla$-Einstein manifolds} With the aim to give a precise definition  of  a $\nabla$-Einstein structure, it is useful to    recall  identities of    the Ricci tensor and the scalar curvature of   a metric connection $\nabla$ with skew-symmetric torsion $0\neq T\in\Lambda^{3}(T^{*}M)$.  We limit ourselves only in a few details and for a general picture  we refer to \cite{FrIv, AFer, Srni}.   As usual, we write $g(\nabla_{X}Y, Z)=g(\nabla^{g}_{X}Y, Z)+\frac{1}{2}T(X, Y, Z)$,     where $\nabla^{g}$ denotes the Levi-Civita connection of the fixed Riemannian manifold $(M^{n}, g)$.   
Let $\{e_{1}, \ldots, e_{n}\}$ be a (local) orthonormal frame of $M$.  In terms of   the co-differential  $\delta T$ and  the  normalized length  $\|T\|^{2}:=(1/6)\sum_{i, j}g(T(e_{i}, e_{j}), T(e_{i}, e_{j}))$  of     $T$,  one has the formulas: $\Sca=\Sca^{g}-\frac{3}{2}\|T\|^{2}$ and
 \[
        \Ric(X, Y)=\Ric^{g}(X, Y)-\frac{1}{4}\sum_{i=1}^{n}g(T(e_{i}, X), T(e_{i}, Y))-\frac{1}{2}(\delta^{g}T)(X, Y).
        \]
        The co-differential  (with respect to $\nabla$) of a $n$-form $\omega$ on $M$   is given  by $\delta^{\nabla}\omega:= -\sum_{i}e_{i}\lrcorner \nabla_{e_{i}} \omega$. For the torsion  3-form it holds that  $\delta^{\nabla}T=\delta^{g}T$.
 We emphasize that the Ricci tensor   of $\nabla$ is not  necessarily  symmetric;  it decomposes into a  symmetric  and antisymmetric part $\Ric=\Ric_{S}+\Ric_{A}$, given by
   \[
   \Ric_{S}(X, Y):=\Ric^{g}(X, Y)-\frac{1}{4}\sum_{i=1}^{n}g(T(e_{i}, X), T(e_{i}, Y)),  \]
 and $\Ric_{A}(X, Y):=-\frac{1}{2}(\delta^{g}T)(X, Y)$, respectively.
   In analogy to  compact  Einstein manifolds, $\nabla$-Einstein manifolds  with skew-torsion   admit a variational approach based on  the  functional 
    \[
    (g, T)\mapsto \int_{M}\Big(\Sca-2\Lambda\Big)d\vol_{g},
    \]
     where $\Lambda$ is a  constant.  In particular, by \cite[Thm.~2.1]{AFer} it is known  that  critical points of this functional are pairs $(g, T)$ as above,  satisfying the equation 
     \[
   -\Ric_{S}+\frac{1}{2}\Sca\cdot g-\Lambda\cdot g=0.
   \]
   For this reason,   one has the following formal definition:
        \begin{Definition}
    We call a  $4$-tuple $(M^{n}, g, \nabla, T)$  a $\nabla$-Einstein manifold with skew-torsion $T$, or in short, a $\nabla$-Einstein manifold,  if the symmetric part of the Ricci tensor of $\nabla$ satisfies the equation $\Ric_{S}=\frac{\Sca}{n}g$.
    \end{Definition}
In contrast to the Riemannian case, for  a $\nabla$-Einstein manifold  the scalar curvature   is not necessarily constant, see \cite{AFer}.  For parallel torsion $T$  one has $\delta^{\nabla}T=0$  and  the Ricci tensor  becomes symmetric  $\Ric=\Ric_{S}$. If in addition $\delta\Ric^{g}=0$, then the scalar curvature is  constant, similarly with  an Einstein manifold. This is the case for any $\nabla$-Einstein manifold $(M, g, \nabla, T)$ with parallel skew-torsion \cite[Prop.~2.7]{AFer}.

  \subsection{The $\nabla^{\al}$-Einstein condition on naturally reductive spaces}\label{NEWPAR}
 From now on   we assume that  $(M^{n}=G/K,  g)$ $(n\geq 3)$ is a  naturally reductive manifold,    endowed with an effective transitive action  of  a connected Lie group $G$ and a reductive decomposition $\fr{g}=\fr{k}\oplus\fr{m}$ such that $\fr{g}=\tilde{\fr{g}}:=\fr{m}+[\fr{m}, \fr{m}]$, see   \cite[p.~569]{Wa1} or \cite[p.~196]{Bes}.
  \begin{remark}\label{afftrans}
\textnormal{Given a homogeneous space $M=G/K$ with a reductive decomposition $\fr{g}=\fr{k}\oplus\fr{m}$,  the ideal $\tilde{\fr{g}}:=\fr{m}+[\fr{m}, \fr{m}]$   is identified with the Lie algebra $\fr{tr}(\nabla^{c})$ of the transvection group ${\rm Tr}(\nabla^{c})$ of the canonical connection associated to    $\fr{m}$ \cite[Rem.~4.1]{Reg2}.   The group ${\rm Tr}(\nabla^{c})$ is a connected and normal subgroup of ${\rm Aff}_{0}(\nabla^{c})$ (the connected component of the affine group of $\nabla^{c}$), which  consists of all $\nabla^{c}$-affine transformations that preserve any $\nabla^{c}$-holonomy subbundle  of the orthonormal frame bundle.  Hence,   our assumption  equivalently says  that $M=G/K$ is a  naturally reductive manifold with respect to the decomposition $\fr{g}=\fr{k}\oplus\fr{m}$,  where $G={\rm Tr}(\nabla^{c})$ is the group of transvections  of the canonical connection $\nabla^{c}$ associated to $\fr{m}$.  In general ${\rm Tr}(\nabla^{c})\subset G$.  However,  if $M=G/K$ is a compact normal homogeneous space, then $G={\rm Tr}(\nabla^{c})$ \cite[Prop.~4.2]{Reg2} and hence any such space gives rise to a (compact) homogeneous manifold satisfying our assumption.}
\end{remark}
 Next we shall  use the $\Ad (G)$-invariant extension $Q$ of the (naturally reductive) inner product $\langle \ , \  \rangle$ on whole $\fr{g}$.  Let  $Q_{\fr{k}}$ be the restriction of $Q$ on $\fr{k}$, i.e. $Q_{\fr{k}}(X, Y)=Q(X_{\fr{k}}, Y_{\fr{k}})$ where $X_{\fr{k}}$ is the $\fr{k}$-component of $X\in\fr{g}$.   A  customary trick  is to associate with $Q_{\fr{k}}$ and the isotropy representation $\chi_{*} : \fr{k}\to\fr{so}(\fr{m})$,  the Casimir element; this is the linear operator  $C_{\chi}\equiv C_{\chi, Q_{\fr{k}}} : \fr{m}\to\fr{m}$ defined by $C_{\chi}:=-\sum_{q}^{\dim\fr{k}}\chi_{*}(k_{q})\circ\chi_{*}(k'_{q})$,    where $\{k_{q}, k'_{q}\}$ are dual bases of $\fr{k}$ with respect to $Q_{\fr{k}}$.   If $\chi$ is an irreducible representation, then $C_{\chi}$ is a scalar operator. 
  We introduce also the symmetric bilinear map $A$ on $\fr{m}$, given by $A(X, Y)=\langle C_{\chi}X, Y\rangle$  for any $X, Y\in\fr{m}$  and  we denote by $B$  the negative of  the restriction of the Killing form of $\fr{g}$ on $\fr{m}$. Then, the following relations are standard  (see for example \cite{Wa1, Bes, Agr03}) 
\begin{equation}\label{use1}
A(X, Y)=\sum_{j}Q_{\fr{k}}([X, Z_{j}], [Y, Z_{j}]),  \quad    B(X, Y)=\sum_{i}\langle [X, Z_{i}]_{\fr{m}}, [Y, Z_{i}]_{\fr{m}}\rangle+2A(X, Y).
\end{equation} 
    Consider  now  the family of $G$-invariant  metric connections $\nabla^{\al}$ described in Proposition \ref{Ilkas}.  
 We   present    formulas for the Ricci tensor and the scalar curvature associated to $\nabla^{\al}$ (see also \cite[Lem.~2.2, Thm.~4.4]{Agr03} for similar expressions).
   \begin{theorem}\label{NEWLIFE}
     The Ricci curvature of the naturally reductive Riemannian manifold $(M=G/K, g)$ endowed with the  family  $\{\nabla^{\al} : \al\in\bb{R}\}$  is given by
  \begin{eqnarray*}
  \Ric^{\al}(X, Y)&=& \frac{1-\al^{2}}{4}\sum_{1\leq i\leq n}\langle [X, Z_{i}]_{\fr{m}}, [Y, Z_{i}]_{\fr{m}}\rangle+A(X, Y)= \frac{1-\al^{2}}{4}B(X, Y)+ \frac{1+\al^{2}}{2}A(X, Y).
  \end{eqnarray*}
  The corresponding scalar curvature $\Sca^{\al} : M\to\bb{R}$  has the form 
  \begin{eqnarray*}
  \Sca^{\al}&=&\frac{1-\al^{2}}{4}\sum_{1\leq i, j\leq n}\|[Z_{i}, Z_{j}]_{\fr{m}}\|^{2}+\sum_{1\leq i, j\leq n}Q_{\fr{k}}([Z_{i}, Z_{j}], [Z_{i}, Z_{j}])\\
  &=&\frac{1-\al^{2}}{4}\sum_{1\leq i, j\leq n}\|[Z_{i}, Z_{j}]_{\fr{m}}\|^{2}+\sum_{1\leq i \leq n}A(Z_{i}, Z_{i}),
  \end{eqnarray*}
  where  $\|Z\|:=\sqrt{\langle Z, Z\rangle}$ is  the norm of a vector $Z\in\fr{m}$ with respect to  $\langle \ , \  \rangle$.
    \end{theorem}
   For a moment,  notice that 
    \[
\Ric^{0}(X, Y)=\frac{1}{4}\sum_{i=1}^{n}\langle [X, Z_{i}]_{\fr{m}}, [Y, Z_{i}]_{\fr{m}}\rangle+A(X, Y)=\frac{1}{4}B(X, Y)+\frac{1}{2}A(X, Y),
\]
 which is the classical formula of the Riemannian Ricci tensor of a   naturally reductive  space $(M=G/K, g)$ with $\fr{g}=\tilde{\fr{g}}$, see \cite[Prop.~1.9, pp.~569--570]{Wa1} or \cite[(7.89b)]{Bes}.  For $G$  compact and semi-simple, one can replace $Q_{\fr{k}}$ with the restriction $B_{\fr{k}}:=B|_{\fr{k}\times\fr{k}}$; then $A(X, Y):=B(C_{\chi, B_{\fr{k}}}X, Y)$  and   in this case it is clear that the Killing metric $g_{B}$ is Einstein if and only if $C_{\chi, B_{\fr{k}}}=\mu\cdot\Id$ for some constant $\mu$. In contrast to the compact case and $C_{\chi, B_{\fr{k}}}$,    the restriction $Q_{\fr{k}}$ is not necessarily positive definite. Thus, in the general case the Casimir operator $C_{\chi, Q_{\fr{k}}}$ can have eigenvalues of either sign;  the same is true for the Ricci tensor and therefore the Einstein condition for the Killing metric is not  anymore equivalent to the relation $C_{\chi, Q_{\fr{k}}}=\mu\cdot\Id$. 
 
 Now,  it is well-known that $\delta^{\al} T^{\al}=0$ for any $\al\in\bb{R}$, see \cite{Agr03}. Therefore, a simple combination with Theorem \ref{NEWLIFE} shows that
     \begin{corollary}\label{symSS}
   The Ricci tensor $\Ric^{\al}$ associated to the family $\nabla^{\al}$ is   symmetric for any $\al\in\bb{R}$, i.e. $\Ric^{\al}\equiv \Ric^{\al}_{S}=\Ric^{g}-\frac{1}{4}S^{\al}$, where $S^{\al}$ is the symmetric tensor  defined by $S^{\al}(X, Y)=\sum_{i}\langle T^{\al}(Z_{i}, X), T^{\al}(Z_{i}, Y)\rangle$. In full details  \[
S^{\al}(X, Y)=\al^{2}\sum_{i}\langle [X, Z_{i}]_{\fr{m}}, [Y, Z_{i}]_{\fr{m}}\rangle=\al^{2}\Big\{B(X, Y)-2A(X, Y)\Big\}.
 \]
 \end{corollary}
We proceed with a  few  remarks for the well-known flat case $R^{\al}=0$, see \cite{DAN, Nagy, FLAT}.  Let  $\Jac_{\fr{m}} : \fr{m}\times\fr{m}\times\fr{m}\to\fr{m}$ be the trilinear map    defined by  $\Jac_{\fr{m}}(X, Y, Z):=\fr{S}^{X, Y, Z}[X, [Y, Z]_{\fr{m}}]_{\fr{m}}$, where $\fr{S}$ denotes the cyclic  sum over the vectors $X, Y, Z\in\fr{m}$.     For a  symmetric space $(M=G/K, g)$ of Type I, it is $\Jac_{\fr{m}}\equiv 0$  identically,   since $[\fr{m}, \fr{m}]\subset\fr{k}$. By   using Theorem \ref{symspace} we  can prove      that  
        \begin{lemma}\label{myjohn}
Let $(M^{n}=G/K, g)$  be a  compact  connected  naturally reductive Riemannian  manifold   endowed  with a $G$-invariant metric connection $\nabla^{\al}$ whose torsion is such that $T^{\al}=\al\cdot T^{c}$, for some $\al\in\bb{R}\backslash\{0, 1\}$. Then the following  are equivalent:

\noindent $(a)$    $Jac_{\fr{m}}\equiv 0$ identically,

 \noindent $(b)$   $[\fr{m}, \fr{m}]\subset\fr{m}$,
 
 \noindent $(c)$ $M^{n}\cong G$  is isometric to a compact Lie group  with a bi-invariant metric.
      \end{lemma} 
\begin{proof}
The less trivial part is the correspondence between $(a)$ and $(b)$.   By  Proposition \ref{Ilkas}  we write $\nabla^{\al}_{X}Y=\nabla^{c}_{X}Y+\Lambda^{\al}(X)Y$, where   $\Lambda^{\al}(X)Y=\frac{1-\al}{2}[X, Y]_{\fr{m}}$  with  $\al\neq 0, 1$. Since $\al\neq 1$, $\nabla^{\al}$ cannot be the canonical connection associated to $\fr{m}$ and by Theorem \ref{symspace}, $M=G/K$ cannot be a  symmetric space of Type I.  First we prove that  the condition $\Jac_{\fr{m}}\equiv 0$  implies the relation $[\fr{m}, \fr{m}]\subset\fr{m}$.  
  In contrast, assume that there exist some $X, Y\in\fr{m}$ such that $[X, Y]\notin\fr{m}$, i.e. $[X, Y]=[X, Y]_{\fr{k}}$. Computing   the   curvature $R^{\al}$ of $\nabla^{\al}$, we see that  
\[
R^{\al}(X, Y)Z= \frac{(1-\al)^{2}}{4}\Jac_{\fr{m}}(X, Y, Z)+\frac{1-\al^{2}}{4}[Z, [X, Y]_{\fr{m}}]_{\fr{m}}-[[X, Y]_{\fr{k}}, Z], 
\]
 which finally reduces to $R^{\al}(X, Y)Z=-[[X, Y]_{\fr{k}}, Z]=-[[X, Y], Z]$, i.e. $R^{\al}$ is identical with the curvature  associated to  the canonical connection.  Then, it is easy to prove that $\nabla^{\al}R^{\al}=0=\nabla^{\al}T^{\al}$,  but    only the canonical connection has this  property (Ambrose-Singer theorem).
Since $\al\neq 1$   we obtain  a contradiction.  Conversely, notice that the relation $\fr{g}=\tilde{\fr{g}}$ implies that   the $\fr{k}$-part of the commutator  $[\fr{m}, \fr{m}]$  spans all of $\fr{k}$. Assuming that  $[\fr{m}, \fr{m}]\subset\fr{m}$, it means that the isotropy algebra $\fr{k}$ is trivial, i.e. $K=\{e\}$ and  $M^{n}\cong G$.  Hence $\fr{g}=\fr{m}$ and by the Jacobi identity we see that $\Jac_{\fr{m}}\equiv 0$. 
\end{proof}
  
In the naturally reductive case, the   flatness condition $R^{\al}\equiv 0$  for some $\al\neq 0$ has as consequence  the parallelism of the associated   torsion form, i.e. $\nabla^{\al}T^{\al}=0$    (see for example  \cite[Prop.~3.7, (d)]{DAN} or \cite[p.~4]{FLAT}). Notice that $(\nabla^{\al}_{Z}T^{\al})(X, Y)=(\al(\al-1)/{2})\Jac_{\fr{m}}(X, Y, Z)$  \cite{Agr03}.   Thus, in the compact case, Lemma \ref{myjohn}  can be used to recover in a Lie theoretic way a part of the classical Cartan-Schouten theorem, namely: {\it If $(M=G/K, g)$ is a de Rham irreducible, compact, connected  naturally reductive manifold as in our assumption, which is flat with respect to the family $\nabla^{\al}$ for some $\al\in\bb{R}\backslash\{0\}$, then $M^{n}\cong G$ is isometric to a compact  simple  Lie group   endowed with a bi-invariant metric and one of the $\pm 1$-connections}. The reason  that  $\Spin(7)/\G_2$  does not appear here,  is due to   our invariant-torsion scenario which  does not allow non-invariant connections, see Remark \ref{notinv}.  
   

 Let us now have  a closer look at  the  $\nabla^{\al}$-Einstein condition on $(M^{n}=G/K, g)$, namely the polynomial equation
 \begin{equation}\label{gen1}
\frac{1-\al^{2}}{4}B(X, Y)+ \frac{1+\al^{2}}{2}A(X, Y)=\frac{\Sca^{\al}}{n}\langle X, Y\rangle.
\end{equation}

 \begin{example}\label{LIEGROUP}  
 \textnormal{
 Consider a  compact connected Lie group $M^{n}\cong G$ with a bi-invariant metric $g$.  Then, the  $\nabla$-Einstein condition is equivalent to the original  Einstein condition \cite[Lemma 2.18]{AFer}.  Indeed, in this case $A$ is identically equal to zero, thus the Ricci tensor associated to $\nabla^{\al}$ is proportional to Riemannian Ricci curvature.   For simplicity, let us use the family $\nabla^{\al}_{X}Y=\eta^{\al}(X, Y):=((1-\al)/2)[X, Y]$.  Then
 \[
 S^{\al}(X, Y)=\al^{2}\sum_{i}\langle [X, Z_{i}], [Y, Z_{i}]\rangle, \quad \Ric^{\al}(X, Y)=\frac{1-\al^{2}}{4}\sum_{i}\langle [X, Z_{i}], [Y, Z_{i}]\rangle,
 \]
  or in other words $\Ric^{\al}=(1-\al^{2})\Ric^{0}$, where $\Ric^{0}\equiv \Ric^{g}$ is the Riemannian Ricci curvature. A little computation also shows that $\Sca^{\al}=\frac{1-\al^{2}}{4}\sum_{i, j}\|[Z_{i}, Z_{j}]\|^{2}$ and   $\|T^{\al}\|^{2}=\frac{\al^{2}}{6}\sum_{i, j}\|[Z_{i}, Z_{j}]\|^{2}$.  Thus, if $g$ is a bi-invariant Einstein metric with Einstein constant $c$, then $G$ is $\nabla^{\al}$-Einstein with (constant) scalar curvature $\Sca^{\al}=n(1-\al^{2})c$  and   torsion $T^{\al}$ such that $6\|T^{\al}\|^{2}=4cn\al^{2}$.  Conversely, if $M^{n}=G$ is $\nabla^{\al}$-Einstein for $\al\neq \pm 1$, then its scalar curvature $\Sca^{\al}$ is constant (since $\nabla^{\al}T^{\al}=0$ for any $\al\in\bb{R}$) and  the bi-invariant metric $g$ is Einstein with Einstein constant $c=\frac{\Sca^{\al}}{n(1-\al^{2})}$.  By Corollary \ref{GENcs} we know that $G$ becomes $\nabla^{\pm 1}$-flat for $\al=\pm 1$, in particular it  is $\Ric^{\pm 1}$-flat  and thus it is trivially a $\nabla$-Einstein manifold, see also  \cite{AFer}. When $G$ is  simple,   the Killing metric $g_{B}=-B$  is  Einstein  with   $c=1/4$  \cite{Wa1}; hence any compact simple Lie group $G$ is a $\nabla^{\al}$-Einstein manifold with constant scalar curvature  $\Sca^{\al}=n(1-\al^{2})/{4}$ and skew-torsion $T^{\al}$ such that $\|T^{\al}\|^{2}=n\al^{2}/{6}$. 
  Because $\SU(2)\cong\Ss^{3}$, we also conclude that the 3-sphere $\Ss^{3}$ is a $\nabla^{\al}$-Einstein manifold with respect to 1-parameter family $\nabla^{\al}$ of bi-invariant metric connections. To summarise:  {\it Any compact connected simple Lie group $G$ is a $\nabla$-Einstein manifold with parallel torsion with respect to a 1-parameter family of bi-invariant metric connections, namely the family $\nabla^{\al}$ described in Theorem \ref{cc}.  In particular, the $\nabla^{\pm 1}$-Einstein structures are flat}.}
\end{example}

Fix now a compact  simply-connected homogeneous Riemannian manifold $(M=G/K, g)$ endowed with an effective $G$-action and   assume that the  isotropy representation of $K$ on $T_{o}G/K\cong\fr{g}/\fr{k}$  is irreducible over $\bb{R}$.  Because  $G$ is compact and its semi-simple part acts transitively on $G/K$,  in addition one can assume that $G$ is a compact, connected, semi-simple Lie group.    If $M$ is an isotropy irreducible homogeneous Riemannian manifold, both  its universal covering and the product  $M\times\cdots\times M$ are isotropy irreducible (the latter  with the product metric).  Hence we shall focus in the case that $M=G/K$  is de Rham irreducible, with the aim to  describe specific solutions of the $\nabla^{\al}$-Einstein condition.   
 Since we are treating invariant connections whose torsion is a multiple of the canonical torsion, by  Theorem \ref{symspace} we have to exclude   symmetric spaces of Type I. 
 On the other hand, notice that $(M=G/K, g)$   admits a unique (up to scale) $G$-invariant (Einstein) metric,   the Killing metric $g_{B}$, see for example \cite[Prop.~7.91, p.~198]{Bes}.  Hence, $(M=G/K, g_{B})$ is  naturally reductive, in particular {\it standard} (in terms of \cite[Def.~7.90]{Bes}); this   means  that $\fr{m}$ is chosen so that $B(\fr{k}, \fr{m})=0$.  
Recall also that basic examples of isotropy irreducible cosets are the strongly isotropy irreducible homogeneous spaces where for  the non-symmetic case, $G$ is always simple \cite[Thm.~1.1, p.~62]{Wolf}.  The isotropy irreducible homogeneous Riemannian manifolds, which are not strongly isotropy irreducible, were classified in \cite{Wa4}.      Finally,  we remark that a non-compact isotropy irreducible space is necessarily symmetric, see \cite[Prop.~7.46]{Bes}. In this case, the form of the family $\nabla^{\al}$  restricts the  $\nabla^{\al}$-Einstein condition to  be valid   only for non-compact simple Lie groups, e.g. ${\rm Sl}(2, \bb{C})$. 
  \begin{theorem}\label{THM4}
Let $(M^{n}=G/K, g_{B})$ be a  compact simply-connected isotropy irreducible standard homogeneous Riemannian manifold, endowed with an effective action of a compact connected simple Lie group $G$. Assume that   $M=G/K$    is not a symmetric space of Type I.  Then, $(M^{n}=G/K, g)$   is a  $\nabla^{\al}$-Einstein manifold for any $\al\in\bb{R}$. The same holds if  $G$ is semi-simple but not simple, i.e. when $M$ is isometric to  the coset $(G/\Delta Z)/(\Delta H/\Delta Z)$ with $G:=H\times \ldots \times H$  ($q$-times) for a compact simply connected simple Lie group $H$. 
  \end{theorem}

\begin{proof}   Assume first that $G$ is simple. 
 By Schur's lemma,   $g_{B}:=B|_{\fr{m}\times\fr{m}}$ is the unique $G$-invariant (Einstein) metric. Consider the Casimir operator $C_{\chi}\equiv C_{\chi, B_{\fr{k}}} : \fr{m}\to\fr{m}$   associated to the isotropy representation $\chi$ and the positive definite restriction $B_{\fr{k}}$. Then, $C_{\chi}={\rm Cas}\cdot \Id_{\fr{m}}$ with $\Cas\in\bb{R}_{+}$.     In particular, given a  $B$-orthonormal basis $\{Z_{1}, \ldots, Z_{n}\}$  of $\fr{m}$, it is $B(C_{\chi} Z_{i}, Z_{i})= B(\Cas \cdot Z_{i}, Z_{i})={\rm Cas}\cdot B(Z_{i}, Z_{i})=\Cas$ and hence
\[
\Cas=A(Z_{i}, Z_{i})=\sum_{j}B_{\fr{k}}([Z_{i}, Z_{j}], [Z_{i}, Z_{j}]),
\]
  where  $A(X, Y):=B(C_{\chi}X, Y)$.  By Theorem \ref{NEWLIFE} it follows that the $\nabla^{\al}$-Einstein  condition (\ref{gen1})  can be expressed by
\begin{equation}\label{mainv}
 \frac{1-\al^{2}}{4}B(X, Y)+ \frac{1+\al^{2}}{2}A(X, Y) = \frac{f(\al)}{4n}B(X, Y), 
\end{equation}
where $f(\al):=(1-\al^{2})\sum_{i, j}\|[Z_{i}, Z_{j}]_{\fr{m}}\|^{2}+4\sum_{i}A(Z_{i}, Z_{i})$ is such that $f(1)=f(-1)=4\sum_{i}A(Z_{i}, Z_{i})=4n\cdot\Cas$. Looking for $\nabla^{\al}$-Einstein structures with skew-torsion for the values $\al=\pm 1$, this formula reduces to 
  \[
  A(X, Y)=\frac{f(\pm 1)}{4n}B(X, Y)=\frac{4n\cdot \Cas}{4n}B(X, Y)=B(\Cas\cdot X, Y)=B(C_{\chi}X, Y),
 \]
which is an identity. In fact,  we will show  that the $\nabla^{\al}$-Einstein condition is an identity for {\it any} $\al\in\bb{R}$. Indeed, (\ref{mainv}) is nothing than the equation $\mu(X, Y)\cdot\al^{2}=\nu(X, Y)$, where  for any  $X, Y\in\fr{m}$  we set
 \begin{eqnarray*}
\mu(X, Y)&:=&2nA(X, Y)+(\sum_{i, j}\|[Z_{i}, Z_{j}]_{\fr{m}}\|^{2}-n)B(X, Y), \\
\nu(X, Y)&:=&-2nA(X, Y)+(\sum_{i, j}\|[Z_{i}, Z_{j}]_{\fr{m}}\|^{2}-n+4n\cdot \Cas)B(X, Y).
\end{eqnarray*} 
Now, it is easy  to see that $\mu(X, Y)\equiv \nu(X, Y)$ and $\mu(X, Y)\equiv 0$, identically.  For example, since $\fr{m}$ is isotropy irreducible, for any $X, Y\in\fr{m}$ it holds that
 \begin{eqnarray*}
 \mu(X, Y)&=&2nA(X, Y)+(\sum_{i, j}\|[Z_{i}, Z_{j}]_{\fr{m}}\|^{2}-n)B(X, Y)\\
 &=&2nB(C_{\chi}X, Y)+(\sum_{i, j}\|[Z_{i}, Z_{j}]_{\fr{m}}\|^{2}-n)B(X, Y)\\
 &=&2n\Cas B(X, Y)+(\sum_{i, j}\|[Z_{i}, Z_{j}]_{\fr{m}}\|^{2}-n)B(X, Y)\\
 &=&(2n\Cas+ \sum_{i, j}\|[Z_{i}, Z_{j}]_{\fr{m}}\|^{2}-n)B(X, Y).
 \end{eqnarray*}
 However, after considering the sum of  the second relation in (\ref{use1})  with respect to  the orthonormal basis $\{Z_{i}\}$ and using the relation $\sum_{i}A(Z_{i}, Z_{i})=n\Cas$, we deduce that $2n\Cas+ \sum_{i, j}\|[Z_{i}, Z_{j}]_{\fr{m}}\|^{2}-n=0$.    
 
 Let us proceed now with the  non-simple case. In \cite[Thm.~2.2]{Wa4} it was shown that the unique example of a de Rham irreducible,  compact,  simply connected isotropy irreducible homogeneous Riemannian manifold $M=G/K$ with $\fr{g}$ semi-simple but not simple, is the coset $G/\Delta H$  where $H$ is a compact  simply connected  simple Lie group and $\Delta H$ the diagonal subgroup of $G:=H\times \ldots \times H$  ($q$-times). The effective version of $G/\Delta H$ has the form $(G/\Delta Z)/(\Delta H/\Delta Z)$, where $Z$ is the finite centre of $H$ and due to our assumptions, $M$ is  isometric to $(G/\Delta Z)/(\Delta H/\Delta Z)$.   The isotropy representation is given by $\chi:=\oplus_{i=1}^{q-1}\ad_{\fr{h}}$ with $\fr{m}:=\{(X_{1}, \ldots, X_{q}) : \sum X_{i}=0, X_{i}\in\fr{h}\}$.   $M$ is a symmetric space  of Type II for $q=2$; for $q\geq 3$  it is isotropy  irreducible but not strongly isotropy irreducible.  The negative $B$ of the Killing form  of $H$ is Einstein, i.e.  $C_{\chi, B}=\Cas\cdot\Id$ for some $\Cas>0$ and our  result above shows  that for any  $q\geq 2$ the standard homogeneous space $M=(G/\Delta Z)/(\Delta H/\Delta Z)$ is a $\nabla^{\al}$-Einstein manifold, for any $\al\in\bb{R}$. For $q=2$, this assertion follows also from the summary in Example \ref{LIEGROUP}. This completes the proof.
    \end{proof}
 Among the  $\nabla^{\al}$-Einstein structures described above, there are two  special members, namely the  structures defined by  the canonical and the anti-canonical connection. These have identical Ricci tensor  $\Ric^{\pm 1}(X, Y)=A(X, Y)=\Cas B(X, Y)$  with $S^{\pm 1}(X, Y)=B(X, Y)-2A(X, Y)=(1-2\Cas)B(X, Y)$ and $\Sca^{\pm 1}=n\Cas>0$. Notice however that  $\nabla^{1}T^{1}=0\neq \nabla^{-1}T^{-1}$ \cite{Agr03}. Moreover, 
  \[
 R^{1}(X, Y, Z)=-[[X, Y]_{\fr{k}}, Z], \quad R^{-1}(X, Y)Z=\Jac_{\fr{m}}(X, Y, Z)-[[X, Y]_{\fr{k}}, Z].\]
Hence, geometrically the Jacobian $\Jac_{\fr{m}}$  represents  the difference of the curvatures associated to  the  canonical and the anti-canonical connection, i.e. $R^{-1}(X, Y, Z)-R^{1}(X, Y, Z)=\Jac_{\fr{m}}(X, Y, Z)$.  Of course, the case   changes   for   a compact simple Lie group $G$, where  the $\nabla^{\pm 1}$-Einstein structures are necessarily flat, i.e. $R^{\pm 1}\equiv 0$  (see  Theorem \ref{cc}, Lemma \ref{myjohn} and Example \ref{LIEGROUP}).   

\begin{example}\textnormal{We conclude  that    the spheres $\Ss^{6}=\G_2/\SU(3)$ and $\Ss^{7}=\Spin(7)/\G_2$ are $\nabla^{\al}$-Einstein manifolds for any $\al\in\bb{R}$.  Notice that the $\nabla^{1}$-Einstein structure on $\Ss^{6}$ coincides with the one induced by the Gray connection (parallel skew-torsion), see \cite{Srni}. On the other hand,    the $\nabla^{\al}$-Einstein structures on $\Ss^{7}$ are  not flat and hence they   differ from the trivial $\nabla^{\rm flat}$-Einstein structure (with non-parallel torsion), associated to the flat connection   discussed  in Remark \ref{notinv}.   For an alternative  study of homogeneous $\nabla$-Einstein structures on odd-dimensional spheres we refer to \cite{Draper}.}
\end{example}
\begin{example}
\textnormal{Let $G$ be a compact simple Lie group whose simple roots are all of the same length; this means that $G$ is one of the groups  $\SU(\ell), \SO(2\ell), \E_6, \E_7, \E_8$. Consider the full flag manifold $M=G/T$, where $T$ is a maximal torus in $G$.  Let $\Delta$ be the root system of the complexification $\fr{g}^{\bb{C}}=\fr{g}\otimes_{\bb{R}}\bb{C}$ with respect to the Cartan subalgebra $\fr{t}^{\bb{C}}=\fr{t}\otimes_{\bb{R}}\bb{C}$, where $\fr{t}$ is the Lie algebra of $T$.  Fix a   Weyl basis $\{E_{\al} : \al\in \Delta\}$ of $\fr{g}^{\bb{C}}$ and for any positive root $\al\in \Delta^{+}$,   set $A_{\al}=E_{\al}+E_{-\al}$, $B_{\al}=i(E_{\al}-E_{-\al})$ and    $\fr{m}_{\al}:=\bb{R}A_{\al}+\bb{R}B_{\al}$.  The  root space decomposition  of $\fr{g}$ has the form $\fr{g} =\fr{t}\oplus\bigoplus_{\al\in \Delta^{+}}\fr{m}_{\al}=\fr{t}\oplus\fr{m}$, i.e. $\fr{m}\cong T_{o}(G/T)=\bigoplus_{\al\in \Delta^{+}}\fr{m}_{\al}$.   Any  $\fr{m}_{\al}$  is a 2-dimensional real irreducible  $\Ad(T)$-module, which simultaneously can be viewed as a complex plane  on which  $T$ acts   by rotations.  The Weyl group of $G$ acts transitively on the factors $\fr{m}_{\al}$ (since acts transitively on roots of the same length)  and by adding  these isometries to $G$,  the full flag manifold $G/T$ becomes an isotropy irreducible space  \cite[p.~235]{Wa4}  with the Killing metric $g_{B}$ being Einstein \cite[Cor.~1.5]{Wa1}.  By Theorem \ref{THM4} it follows that these  full flag manifolds  are also  $\nabla^{\al}$-Einstein manifolds for any $\al\in\bb{R}$.  }
\end{example}
\begin{example}
\textnormal{Consider the 7-dimensional Berger sphere $B^{7}=\SO(5)/\SO(3)_{\rm ir}$.  The embedding   $\SO(3)_{\rm ir}\subset\SO(5)$ is given by  the  unique 5-dimensional $\SO(3)$-irreducible representation which is defined by the action of $\SO(3)$ on the set of $3\times 3$ symmetric traceless matrices $\Ss^{2}_{0}(\bb{R}^{3})\cong\bb{R}^{5}$,   via conjugation.   The isotropy representation coincides with the unique 7-dimensional $\SO(3)$-irreducible representation, and this  defines an embedding of $\SO(3)$ inside $\G_2$ and thus a $\G_2$-structure.  The  Killing metric $g_{B}$ is the unique $\SO(5)$-invariant (Einstein) metric (up to scale).  By Theorem \ref{THM4}  we conclude that  $B^{7}$ it is also  a $\nabla^{\al}$-Einstein manifold for any $\al\in\bb{R}$.  In fact, the $\G_2$-structure is (proper) nearly parallel,  i.e.  the coset $\SO(5)/\SO(3)_{\rm ir}$ admits a unique real Killing spinor \cite{FKMS}.  Hence,  the canonical connection $\nabla^{1}$ coincides with the  characteristic connection $\nabla^{c}$ preserving this  structure, in particular the $\nabla^{1}$-Einstein structure is well-known, see \cite[p.~318]{FrIv}. }
\end{example}

   \section{Homogeneous Riemannian manifolds with two isotropy summands}\label{MINES}

   In this final section we study compact connected homogeneous Riemannian manifolds   $(M=G/K, g)$   whose isotropy representation $\chi : K\to\SO(\fr{m})$ decomposes   into two (non-trivial)   inequivalent  and irreducible  $K$-submodules satisfying   (\ref{incl}).   Well-known examples are:     connected semi-simple Lie groups,     flag manifolds with two isotropy summands, odd-dimensional spheres,   3-- and 4--symmetric spaces (see \cite{Chry, Chry1} and the references therein).  Without loss of generality we  assume that the compact Lie group  $G$ is    connected and semi-simple  and  that  $K$ is connected, see \cite{Bes, Wa1}.  We consider the 1-parameter family of   $G$-invariant Riemannian metrics on $M=G/K$, given by $g_{t}=B|_{\fr{m}_{1}\times\fr{m}_{1}}+2t\cdot B|_{\fr{m}_{2}\times\fr{m}_{2}}$ for some $t\in\bb{R}_{+}$,
where $B$ denotes the negative of the Killing form of $\fr{g}$.  It follows that any  $G$-invariant Riemannian metric on $M=G/K$ is  a multiple of  $g_{t}$.  The value $t=1/2$  defines  the Killing metric   $g_{1/2}=g_{B}=B|_{\fr{m}_{1}\times\fr{m}_{1}}+ B|_{\fr{m}_{2}\times\fr{m}_{2}}$. 
      \begin{remark}\label{symspaces}
  \textnormal{There is a natural construction that gives rise to compact homogeneous spaces with two isotropy summands satisfying (\ref{incl}). Consider a semi-simple Lie algebra  $\fr{g}$ and  assume  that the pairs $(\fr{g}, \fr{k}\oplus\fr{m}_{2})$ and  $(\fr{k}\oplus\fr{m}_{2}, \fr{k})$ are orthogonal symmetric pairs such that  $\fr{m}_{1}$ be an orthogonal complement of $\fr{k}\oplus\fr{m}_{2}$ in $\fr{g}$, with respect to    the Killing form of $\fr{g}$. Then, by setting $\fr{m}=\fr{m}_{1}\oplus\fr{m}_{2}$ one can easily verify  the inclusions given by (\ref{incl}).  In this way we obtain a  Riemannian submersion $U/K\to G/K\to G/U$ where $U$ is the connected Lie group generated by the Lie algebra $\fr{u}=\fr{k}\oplus\fr{m}_{2}$, and   both (effective) quotients $U/K$ and $G/U$ are symmetric spaces.  If $\fr{p}$ denotes an orthogonal complement of $\fr{u}$ in $\fr{g}$, i.e. $\fr{g}=\fr{u}\oplus\fr{p}$, then we may identify $\fr{p}=T_{o}G/U=\fr{m}_{1}$, $\fr{m}_{2}=T_{o}U/K$ and $\fr{m}=T_{o}G/K=\fr{m}_{1}\oplus\fr{m}_{2}$. }
  \end{remark}

  Next we shall characterize   the $G$-invariant metric connections of $M=G/K$ which have (totally) skew-symmetric torsion.  
   First we recall  some details.
    Consider a connected Riemannian  manifold $(M^{n}, g)$  carrying  a  metric  connection $\nabla$.   Via $g$ we shall identify $TM\cong T^{*}M$.   Set       ${\cal A}^{g}:=\{A\in\otimes^{3}TM :   A(X, Y, Z)+A(X, Z, Y)=0\}\cong TM\otimes\Lambda^{2}(TM)$.      Since $\nabla$ is metric  we write $\nabla_{X}Y-\nabla^{g}_{X}Y=A(X, Y)$ for a $(2, 1)$-tensor field $A\in{\cal{A}}^{g}$.  Given a (local) orthonormal frame  $\{e_{i}\}$ of $M$, let $\Phi : {\cal{A}}^{g}\to T^{*}M$ the map defined by $\Phi(A)(Z)=\sum_{i}A(e_{i}, e_{i}, Z)$,          where $A(X, Y, Z):=g(A(X, Y), Z)$.  It is   well-known  that ${\cal{A}}^{g}$ coincides with the space of torsion tensors and under the action of $\Oo(n)$  it decomposes  into three irreducible  representations, i.e. ${\cal{A}}^{g}={\cal{A}}_{1}\oplus{\cal{A}}_{2}\oplus{\cal{A}}_{3}$ (see for example \cite{Srni, Pfa}). These are  explicitly  given by
                  \begin{eqnarray*}
         {\cal{A}}_{1}&:=&\{A\in{\cal{A}}^{g} : A(X, Y, Z)=g(X, Y)g(V, Z)-g(V, Y)g(X, Z), \ V\in{\cal{X}}(M)\},\\
         {\cal{A}}_{2}&:=&\{A\in{\cal{A}}^{g} : A(X, Y, Z)+A(Y, X, Z)=0\},\\
         {\cal{A}}_{3}&:=&\{A\in{\cal{A}}^{g} :  \fr{S}^{X, Y, Z}A(X, Y, Z)=0, \Phi(A)=0\}.
         \end{eqnarray*}
            We say that the torsion $T^{\nabla}(X, Y)=\nabla_{X}Y-\nabla_{Y}X-[X, Y]$ is of vectorial type (and the same   for  $\nabla$)  if  $A\in{\cal{A}}_{1}\cong TM$, (totally) skew-symmetric if $A\in{\cal{A}}_{2}\cong \Lambda^{3}(T^{*}M)$ and finally of Cartan type   if  $A\in{\cal{A}}_{3}$. For $n=2$, ${\cal{A}}^{g}\cong\bb{R}^{2}$ is $\Oo(2)$-irreducible. 
  
                    \begin{theorem}\label{chato}
   Let $\nabla$ be a    $G$-invariant metric connection  of  the homogeneous Riemannian manifold $(M^{n}=G/K, \fr{m}_{1}\oplus\fr{m}_{2}, g_{t})$ $(n\geq 3)$ with non-trivial  torsion $T\neq  0$.  Then, $T$ is totally skew-symmetric if and only if $t=1/2$  and  $\Lambda(X)X=0$ for any $X\in\fr{m}$, where $\Lambda : \fr{m}\to\fr{so}(\fr{m})$ denotes the associated Nomizu map. 
        \end{theorem}
     \begin{proof}
   The  tensor field $T(X, Y, Z)=g_{t}(T(X, Y), Z)$  which occurs  from $T$ by contraction with $g_{t}$  is already skew-symmetric with respect  to $X, Y$.  Set $\fr{T}(X, Y, Z):=T(X, Y, Z)+T(X, Z, Y)$. Then, the condition $T\in\Lambda^{3}(\fr{m})$   is equivalent to    $\fr{T}(X, Y, Z)=0$,  for any  $X, Y, Z\in\fr{m}$.
  We compute
                  \begin{eqnarray*}
\fr{T}(X, Y, Z)&=&\underbrace{g_{t}(\Lambda(X)Y, Z)}_{(\al')}-\underbrace{g_{t}(\Lambda(Y)X,  Z)}_{(\be')}-\underbrace{g_{t}([X, Y]_{\fr{m}}, Z)}_{(\gamma')}\\
&&+\underbrace{g_{t}(\Lambda(X)Z, Y)}_{(\delta')}-\underbrace{g_{t}(\Lambda(Z)X,  Y)}_{(\varepsilon')}-\underbrace{g_{t}([X, Z]_{\fr{m}}, Y)}_{(\zeta')},
\end{eqnarray*}
for  any $X, Y, Z\in\fr{m}$. Since $\Lambda(X)\in\fr{so}(\fr{m})$ for all $X\in\fr{m}$, we see  that the terms  $(\al')$ and $(\delta')$ cancel one another: $g_{t}(\Lambda(X)Y, Z)+g_{t}(\Lambda(X)Z, Y)=-g_{t}(\Lambda(X)Z, Y)+g_{t}(\Lambda(X)Z, Y)=0$.  For the same reason  it is
\begin{eqnarray*}
 -(\be')-(\varepsilon')&=&-g_{t}(\Lambda(Y)X,  Z)-g_{t}(\Lambda(Z)X,  Y)=g_{t}(\Lambda(Y)Z, X)+g_{t}(\Lambda(Z)Y, X)\\
 &=&g_{t}(\Lambda(Y)Z+\Lambda(Z)Y, X).
\end{eqnarray*}
Hence,    the equation  $\fr{T}(X, Y, Z)=0$ becomes equivalent to  
      \begin{equation*}\label{clev}
     \underbrace{g_{t}(\Lambda(Y)Z+\Lambda(Z)Y, X)}_{(\al)}-\underbrace{g_{t}([X, Y]_{\fr{m}}, Z)}_{(\be)}
-\underbrace{g_{t}([X, Z]_{\fr{m}}, Y)}_{(\gamma)}=0, \quad \forall   \  X, Y, Z\in\fr{m}.
      \end{equation*}
 By using this formula one needs   to examine each possible case separately.  Consider for example some non-zero vectors  $X, Y\in\fr{m}_{1}$  and $Z\in\fr{m}_{2}$.  Then 
 \begin{eqnarray*}
-(\be)-(\gamma)&=&g_{t}([Y, X]_{\fr{m}_{2}}, Z)+g_{t}([Z, X], Y)=2tB([Y, X]_{\fr{m}_{2}}, Z)+B([Z, X], Y)\\
&=&-(2t-1)B([Y, Z], X)=-g_{t}((2t-1)[Y, Z], X).
\end{eqnarray*}
Hence the skew-symmetry of $T$ reduces to the equation $g_{t}(\Lambda(Y)Z+\Lambda(Z)Y-(2t-1)[Y, Z], X)=0$, for any  $X\in\fr{m}_{1}$, $Y\in\fr{m}_{1}$ and $Z\in\fr{m}_{2}$, which means that $\Lambda(Y)Z+\Lambda(Z)Y=(2t-1)[Y, Z]$, for all $Y\in\fr{m}_{1}, Z\in\fr{m}_{2}$.
  Let now $X\in\fr{m}_{2}$, $Y\in\fr{m}_{1}$ and $Z\in\fr{m}_{2}$.   Then we get  $\Lambda(Y)Z +\Lambda(Z)Y=0$, for any  $Y\in \fr{m}_{1}, Z\in\fr{m}_{2}$, and this gives rise to the following system of  equations
  \[
\Big\{\Lambda(Y)Z+\Lambda(Z)Y=(2t-1)[Y, Z], \quad  \Lambda(Y)Z +\Lambda(Z)Y=0, \quad \forall \ Y\in\fr{m}_{1},   Z\in\fr{m}_{2}\Big\}.
\]
Thus,   it must be $t=1/2$ and $\Lambda(Y)Z +\Lambda(Z)Y=0$,  for all  $Y\in\fr{m}_{1}$, $Z\in\fr{m}_{2}$. One can obtain the same result   by comparing the cases  $X\in\fr{m}_{1}, Y\in\fr{m}_{2}, Z\in\fr{m}_{1}$ and $X, Y\in\fr{m}_{2}, Z\in\fr{m}_{1}$, respectively, i.e. $t=1/2$ and $\Lambda(Y)Z +\Lambda(Z)Y=0$, for any $Y\in\fr{m}_{2}$, $Z\in\fr{m}_{1}$.  Similar are treated the other cases.   The converse  direction  follows by  \cite[Lem.~2.1]{Agr03},  since for    $t=1/2$ we obtain the Killing metric $g_{1/2}$ which is  naturally reductive.
\end{proof}

 \begin{proposition}\label{wang}\textnormal{(\cite[Lem.~10, p.~141]{Baum})}
    The Nomizu map $\Lambda_{t} : \fr{m}\to\fr{so}(\fr{m})$ associated to the Levi-Civita  connection $\nabla^{t}(\equiv\nabla^{g_{t}})$ $(t>0)$ of  the homogeneous Riemannian manifold  $(M=G/K,  \fr{m}_{1}\oplus\fr{m}_{2}, g_{t})$  is defined by the following relations
      \[
    \begin{tabular}{ll}
     $\Lambda_{t}(\fr{m}_{1})\fr{m}_{1}=(1/2)[\fr{m}_{1}, \fr{m}_{1}]_{\fr{m}_{2}}$, & $\Lambda_{t}(\fr{m}_{2})\fr{m}_{1}=(1-t)[\fr{m}_{2}, \fr{m}_{1}]$, \\
     $\Lambda_{t}(\fr{m}_{1})\fr{m}_{2}=t[\fr{m}_{1}, \fr{m}_{2}]$, & $\Lambda_{t}(\fr{m}_{2})\fr{m}_{2}=0$.
     \end{tabular}
    \]
   \end{proposition} 
We fix finally some notation that will be used throughout this section. We shall denote by   $D_{i, j}$   the $(n\times n)$-matrix having $1$ in the $(i, j)$-entry and zeros elsewhere and we set   $E_{i, j}=-D_{i, j}+D_{j, i}$, with $1\leq i\neq j\leq n$. We also set $d_{i}:=\dim\fr{m}_{i}$ for any    $i=1, 2$ and  fix a $B$-orthonormal  basis of $\fr{m}$ adapted to the decomposition $\fr{m}=\fr{m}_{1}\oplus\fr{m}_{2}$, that is
$\{0\neq X_{i}\in\fr{m}_{1} : 1\leq i\leq d_{1}\}\sqcup\{0\neq Y_{k}\in\fr{m}_{2} : 1\leq k\leq d_{2}\}$,  such that $\fr{m}_{1}=\Span_{\bb{R}}\{X_{i}\}_{i=1}^{d_{1}}$, $\fr{m}_{2}=\Span_{\bb{R}}\{Y_{k}\}_{k=1}^{d_{2}}$, with $B(X_{i}, X_{j})=\delta_{i, j}$, $B(Y_{k}, Y_{l})=\delta_{k, l}$,  $B(X_{i}, Y_{k})=0$.  The associated  $g_{t}$-orthonormal bases are of the form
  \[
\fr{m}_{1}=\Span_{\bb{R}}\{V_{i}:=X_{i} : 1\leq i\leq d_{1}\}, \quad  \fr{m}_{2}=\Span_{\bb{R}}\{W_{k}:=Y_{k}/\sqrt{2t} : 1\leq k\leq d_{2}\}.
\]

 \subsection{A new family of $G$-invariant metric connections}
For a  new parameter $s\in\bb{R}$   consider the map
 $\Lambda^{\fr{m}}_{s, t} \equiv \Lambda_{s, t} : \fr{m}\to\fr{so}(\fr{m})$ with $\Lambda_{s, t}(X)Y:=s\cdot \Lambda_{t}(X)Y$, i.e.
   \begin{equation}\label{newc}
    \begin{tabular}{ll}
     $\Lambda_{s, t}(\fr{m}_{1})\fr{m}_{1}=(s/2)[\fr{m}_{1}, \fr{m}_{1}]_{\fr{m}_{2}}$, & $\Lambda_{s, t}(\fr{m}_{2})\fr{m}_{1}=s(1-t)[\fr{m}_{2}, \fr{m}_{1}]$, \\
     $\Lambda_{s, t}(\fr{m}_{1})\fr{m}_{2}=st[\fr{m}_{1}, \fr{m}_{2}]$, & $\Lambda_{s, t}(\fr{m}_{2})\fr{m}_{2}=0$.
     \end{tabular}
     \end{equation}
   Obviously, $\Lambda_{s, t}$ is  an  $\Ad(K)$-equivariant linear map such that   $\Lambda_{s, t}(X)\in\fr{so}(\fr{m})$ for any $X\in\fr{m}$. Thus, it induces  a   2-parameter family  of $G$-invariant metric connections $\{\nabla^{s, t} :   s\in\bb{R}, \  t\in\bb{R}_{+}\}$, which after identifying $\fr{m}=T_{o}G/K$, can be explicitly  described by
                   \begin{equation}\label{newcon}
      \nabla^{s, t}_{X}Y:= \nabla^{0}_{X}Y+\Lambda_{s, t}(X)Y=\nabla^{0}_{X}Y+s\cdot \Lambda_{t}(X)Y, \quad   \forall \ X, Y\in\fr{m}.      \end{equation}
   Notice that $\nabla^{s, t}$   joins  the canonical connection  $\nabla^{0, t}\equiv  \nabla^{0}\equiv\nabla^{c}$ $(s=0)$  and  the  Levi-Civita connection $\nabla^{1, t}\equiv\nabla^{t}$ $(s=1)$.    Now, using  (\ref{torsion}) it follows that
   \begin{lemma}\label{T}
      The torsion $T^{s, t}$   is given as follows:
   \[
   \begin{tabular}{ll}
   $T^{s, t}(\fr{m}_{1}, \fr{m}_{1})=(s-1)[\fr{m}_{1}, \fr{m}_{1}]_{\fr{m}_{2}}$, & $T^{s, t}(\fr{m}_{2}, \fr{m}_{1})=(s-1)[\fr{m}_{2}, \fr{m}_{1}]$, \\
   $T^{s, t}(\fr{m}_{1}, \fr{m}_{2})=(s-1)[\fr{m}_{1}, \fr{m}_{2}]$, & $T^{s, t}(\fr{m}_{2}, \fr{m}_{2})=0$.
   \end{tabular}
   \]
      Consequently, the  $3$-tensor $T^{s, t}(X, Y, Z):=g_{t}(T^{s, t}(X, Y), Z)$ is such that
  \begin{equation}\label{tor1}
  \begin{tabular}{l}
  $T^{s, t}(\fr{m}_{1}, \fr{m}_{1}, \fr{m}_{2})=2t(s-1)B([\fr{m}_{1}, \fr{m}_{1}]_{\fr{m}_{2}}, \fr{m}_{2})$,\\
  $T^{s, t}(\fr{m}_{2}, \fr{m}_{1}, \fr{m}_{1})=(s-1)B([\fr{m}_{2}, \fr{m}_{1}], \fr{m}_{1}),$ \\
  $T^{s, t}(\fr{m}_{1}, \fr{m}_{2}, \fr{m}_{1})=(s-1)B([\fr{m}_{1}, \fr{m}_{2}], \fr{m}_{1}),$ 
    \end{tabular}
  \end{equation}
     \noindent  and   all the other combinations are zero.
     \end{lemma}
   Although $T^{s}(X, Y)=-(s-1)T^{0}(X, Y)$, the induced 3-tensor  $T^{s, t}(X, Y, Z)$ is not a 3-form for any $s, t$;    Theorem \ref{chato} states  that  this is   possible only for $t=1/2$ and under the further assumption $\Lambda_{s, t}(X)X=0$ for   any $X\in\fr{m}$. Writing $\fr{m}\ni X=V+Z$ with $V\in\fr{m}_{1}, Z\in\fr{m}_{2}$,  the  latter condition reduces to $s(2t-1)[V, Z]=0$. This is always satisfied for  $t=1/2$, i.e. $\Lambda_{s, \frac{1}{2}}(X)X=0$ for any  $X\in\fr{m}$, as required.   In fact, by (\ref{tor1}) and for  non-zero vectors $X\in\fr{m}_{1}$, $Y\in\fr{m}_{1}$ and $Z\in\fr{m}_{2}$, it follows that the equation     $T^{s, t}(X, Y, Z)+T^{s, t}(X, Z, Y)=0$ is equivalent to  $(s-1)(2t-1)B([X, Y]_{\fr{m}_{2}}, Z)=0$.      Consequently, there are two possibilities: $s=1$ or $t=1/2$. The first   fails, since  it   corresponds to the Riemannian connection. Hence $t=1/2$ which corresponds to the Killing metric.  Similarly are treated the other cases.  We conclude that
        \begin{corollary}\label{NEWTORSION}
 For any $s\neq 1$  it holds that $0\neq T^{s, t}(X, Y, Z)\in\Lambda^{3}(\fr{m})$ for any $X, Y, Z\in\fr{m}$  $\Leftrightarrow$  $t=1/2$.
 \end{corollary}
 Let us describe now the algebraic  type of the torsion $T^{s, t}$  for any $s\in\bb{R}, t>0$.  It is useful to present the  endomorphism $\Lambda^{s, t}(X)\in\fr{so}(\fr{m})$ $(X\in\fr{m})$ in terms of the matrices $E_{i, j}$. This is given by 
         \begin{eqnarray*}
 (a) \quad \Lambda^{s, t}(X)&=&\frac{st}{\sqrt{2t}}\sum_{i, k}B([X, X_{i}]_{\fr{m}_{2}}, Y_{k})E_{i, k}, \quad \text{for} \ \ X\in\fr{m}_{1},\\
 (b) \quad \Lambda^{s, t}(X)&=&\frac{s(1-t)}{2}\sum_{i, j}B([X, X_{i}], X_{j})E_{i, j}, \quad \text{for} \ \ X\in\fr{m}_{2}.
 \end{eqnarray*}


  \begin{theorem}
For any $s\in\bb{R}$ and  $t\in\bb{R}_{+}-\{1/2\}$, the 2-parameter family of  $G$-invariant metric connections $\nabla^{s, t}$ has torsion of mixed type ${\cal{A}}_{2}\oplus{\cal{A}}_{3}$. For $t=1/2$  it reduces to 1-parameter family   of metric connections with skew-torsion. 
 \end{theorem}
 \begin{proof}
{\noindent \bf 1st way.}  It is known that  given a  connected Riemannian manifold $(M^{n}, g)$  endowed with a metric connection $\nabla$, then the condition $A:=\nabla-\nabla^{g}\in{\cal{A}}_{2}\oplus{\cal{A}}_{3}$ is equivalent to say that $\Phi(A)(X)=0$ for  any vector field $X$ on $M$, or in other words  that the $\nabla$-divergence of   $X$   coincides with the Riemannian divergence \cite[Corol.~4.6]{Pfa}.  Therefore, it is sufficient to show that $\Div^{s, t}(X)-\Div^{t}(X)=0$, for any $X\in\fr{m}$, where $\Div^{s, t}$ and $\Div^{t}$ are the divergences with respect to $\nabla^{s, t}$ and  $\nabla^{t}$, respectively. Because $\fr{m}=\fr{m}_{1}\oplus\fr{m}_{2}$ is an orthogonal splitting, we write   $X=X^{\al}+X^{\be}$  for some $X^{\al}\in\fr{m}_{1}$ and $X^{\be}\in\fr{m}_{2}$.  Set $\fr{D}(X):=\Div^{s, t}(X)-\Div^{t}(X)$. Since  $A^{s, t}:=\nabla^{s, t}-\nabla^{t}=(s-1)\Lambda^{t}$,  an easy computation proves our claim:
\begin{eqnarray*}
\fr{D}(X)&=& \big(\Div^{s, t}(X^{\al})-\Div^{t}(X^{\al})\big)+\big(\Div^{s, t}(X^{\be})-\Div^{t}(X^{\be})\big)\\
 &=&\sum_{1\leq i \leq d_{1}}g_{t}(X_{i}, \nabla^{s, t}_{X_{i}}X^{\al}-\nabla^{t}_{X_{i}}X^{\al})+\sum_{1\leq k \leq d_{2}}g_{t}(W_{k}, \nabla^{s, t}_{W_{k}}X^{\al}-\nabla^{t}_{W_{k}}X^{\al})\\
&&+ \sum_{1\leq i \leq d_{1}}g_{t}(X_{i}, \nabla^{s, t}_{X_{i}}X^{\be}-\nabla^{t}_{X_{i}}X^{\be})+\sum_{1\leq k \leq d_{2}}g_{t}(W_{k}, \nabla^{s, t}_{W_{k}}X^{\be}-\nabla^{t}_{W_{k}}X^{\be})\\
&=& \sum_{i}g_{t}(X_{i}, (s-1)\Lambda^{t}(X_{i})X^{\al})+\sum_{k}g_{t}(W_{k}, (s-1)\Lambda^{t}(W_{k})X^{\al})\\
&&+ \sum_{i}g_{t}(X_{i}, (s-1)\Lambda^{t}(X_{i})X^{\be})+\sum_{k}g_{t}(W_{k}, (s-1)\Lambda^{t}(W_{k})X^{\be})\\
&=& \sum_{i}g_{t}(X_{i}, \frac{(s-1)}{2}[X_{i}, X^{\al}]_{\fr{m}_{2}})+\sum_{k}g_{t}(W_{k}, (s-1)(1-t)[W_{k}, X^{\al}])\\
&&+\sum_{i}g_{t}(X_{i}, (s-1)t[X_{i}, X^{\be}])\\
&\overset{\fr{m}_{1}\perp\fr{m}_{2}}{=}&t(s-1)\sum_{i}B(X_{i}, [X_{i}, X^{\be}])=-t(s-1)\sum_{i}B([X_{i}, X_{i}]_{\fr{m}_{2}}, X^{\be})=0.
\end{eqnarray*}
   {\noindent \bf 2nd way.}     Given a   Riemannian spin manifold $(M^{n}, g)$ with a metric connection $\nabla$,  the  Dirac operator  $D\varphi:=\sum_{i}e_{i}\cdot\nabla_{e_{i}}\varphi$  associated to $\nabla$ is formally self-adjoint if and only if $\Div^{\nabla}(X)=\Div^{g}(X)$ for any vector field $X$ \cite[Satz.~2]{FS}. Hence, after   assuming that $(M^{n}=G/K, g_{t})$ carries  an invariant spin structure, one can show that  $\nabla^{s, t}$ has no vectorial component, i.e. $A^{s, t}\in{\cal{A}}_{2}\oplus{\cal{A}}_{3}$ by proving the identification $(D^{s, t})^{*}=D^{s, t}$, where $D^{s, t}$ is the Dirac operator associated to the 2-parameter family $\nabla^{s, t}$.  A $G$-invariant spin structure on $M=G/K$ corresponds to a lift  of the isotropy representation $\chi$  into    the spin group $\Spin(\fr{m})$, i.e. a homomorphism  $\widetilde{\chi} : K\to\Spin(\fr{m})$ such that  $\chi=\lambda\circ\widetilde{\chi}$, where $\lambda : \Spin(\fr{m})\to\SO(\fr{m})$ is the double covering.  Because the tangent bundle splits $T(G/K)=(G\times_{\chi^{1}}\fr{m}_{1})\oplus(G\times_{\chi^{2}}\fr{m}_{2})$, the  Clifford algebra $(\Cl(\fr{m}), B)$ of $\fr{m}=\fr{m}_{1}\oplus\fr{m}_{2}$ with respect to $B$ is the graded tensor product $\Cl(\fr{m})=\Cl(\fr{m}_{1}\oplus\fr{m}_{2})=\Cl(\fr{m}_{1})\hat{\otimes}\Cl(\fr{m}_{2})$,   where  $\chi^{i} : K\to\SO(\fr{m}_{i})$ and $\Cl(\fr{m}_{i})$ denote the Clifford algebras of $\fr{m}_{i}$ with respect to the inner products $B_{\fr{m}_{i}}=B|_{\fr{m}_{i}\times\fr{m}_{i}}$, for $i=1, 2$.  We shall denote by $\kappa_{n} : \Cl^{\bb{C}}(\fr{m})\stackrel{\sim}{\rightarrow}\Ed(\Delta_\fr{m})$  the Clifford representation and by     $\mu(X\otimes \phi):=\kappa_{n}(X)\psi=X\cdot \psi$  the Clifford multiplication between vectors and spinors,  see \cite{Agr03, Srni, Baum} for more details.   Set $\rho:=\kappa\circ \widetilde{\chi} : K\to\Aut(\Delta_{\fr{m}})$, where  $\kappa=\kappa_{n}|_{\Spin(\fr{m})} : \Spin(\fr{m})\to\Aut(\Delta_{\fr{m}})$  is the   spin representation.  The spinor bundle $\Sigma\to G/K$ is the homogeneous  vector bundle  associated to the    $\Spin(\fr{m})$-principal bundle   $P:=G\times_{\widetilde{\chi}}\Spin(\fr{m})$ via the representation $\rho$, i.e. $\Sigma=G\times_{\rho}\Delta_{\fr{m}}$.   Therefore  we may   identify sections of $\Sigma$ with smooth functions $\varphi : G\to\Delta_{\fr{m}}$ such that  
 \[
 \varphi(gk)=\kappa\big(\widetilde{\chi}(k^{-1})\big)\varphi(g)=\rho(k^{-1})\varphi(g), \quad \forall \ g\in G, k\in K.
 \]
It is useful to fix spin endomorphisms  induced by the $B$-orthonormal vectors $\{X_{i}\}_{i=1}^{d_{1}}$, $\{Y_{k}\}_{k=1}^{d_{2}}$ of $\fr{m}_{1}, \fr{m}_{2}$  and interpret the  Clifford relations  by 
   \begin{equation} \label{Cliffm}
 \begin{tabular}{lllll}
  $\kappa_{n}({X}_{i})\kappa_{n}(X_{j})+\kappa_{n}(X_{j})\kappa_{n}(X_{i})$&=&$X_{i}\cdot X_{j}+X_{j}\cdot X_{i}$&=&$-2\delta_{i, j}$,\\
 $\kappa_{n}(Y_{k})\kappa_{n}(Y_{l})+\kappa_{n}(Y_{l})\kappa_{n}(Y_{k})$&=& $Y_{k}\cdot Y_{l}+Y_{l}\cdot Y_{k}$&=&$-2\delta_{k, l}$, \\
  $\kappa_{n}(X_{i})\kappa_{n}(Y_{k})+\kappa_{n}(Y_{k})\kappa_{n}(X_{i})$&=& $ X_{i}\cdot Y_{k}+Y_{k}\cdot X_{i}$&=& 0.
  \end{tabular}
 \end{equation}
   Due to the definition of $\lambda_{*}$ it is also $\lambda_{*}^{-1}(E_{i, j})=  (X_{i}\cdot X_{j})/2$, $\lambda_{*}^{-1}(E_{k, l})=  (Y_{k}\cdot Y_{l})/{2}$ and $\lambda_{*}^{-1}(E_{i, k})=(X_{i}\cdot Y_{k})/{2}$. 
 Consider now the lift  $\widetilde{\Lambda_{s, t}}:=\lambda_{*}^{-1}\circ \Lambda_{s, t}: \fr{m}\to\fr{spin}(\fr{m})$ of the Nomizu map   into $\fr{spin}(\fr{m})$.  
 Then,  by using the expression of  $\Lambda^{s, t}(X)$ described above, we conclude  that
  \begin{equation} \label{liftla1}
\begin{tabular}{lcl}
  $(a)$  If  $X\in\fr{m}_{1},$  & then  & $\widetilde{\Lambda_{s, t}}(X)=\displaystyle\frac{st}{2\sqrt{2t}}\sum_{i, k}B([X, X_{i}]_{\fr{m}_{2}}, Y_{k})X_{i}\cdot Y_{k}$,\\
 $(b)$ If  $X\in\fr{m}_{2},$ & then  &   $\widetilde{\Lambda_{s, t}}(X)=\displaystyle\frac{s(1-t)}{4}\sum_{i, j}B([X, X_{i}], X_{j})X_{i}\cdot X_{j}$.
 \end{tabular} 
  \end{equation}
  The lifted connection on spinor fields reads by (we use the same notation) $\nabla^{s, t}_{X}\varphi=X(\varphi)+\widetilde{\Lambda_{s, t}}(X)\varphi$  and  the action of the Dirac operator on $\Gamma(\Sigma)$ is given by  (see   \cite{Agr03, Baum} for the used notation), 
 \begin{eqnarray*}
 D^{s, t}(\varphi) &=&  \sum_{i=1}^{d_{1}}\kappa_{n}(X_{i})\big\{X_{i}(\varphi) + \widetilde{\Lambda_{s, t}}(X_{i})\varphi\big\}+\sum_{k=1}^{d_{2}}\kappa_{n}(Y_{k}) \big\{\frac{Y_{k}}{\sqrt{2t}}(\varphi) +   \widetilde{\Lambda_{s, t}}(\frac{Y_{k}}{\sqrt{2t}})\varphi\big\}.
 \end{eqnarray*}
 Using the relations (\ref{Cliffm}) and (\ref{liftla1}), it finally takes the form 
 \[
D^{s, t}(\varphi)= D^{0}(\varphi)+\frac{s(1+t)}{4\sqrt{2t}}\sum_{i, j, k}B([X_{i}, X_{j}]_{\fr{m}_{2}}, Y_{k})X_{i}\cdot X_{j}\cdot Y_{k}\cdot\varphi,
 \]
  where $D^{0}(\varphi):=\sum_{1\leq 1\leq d_{1}}X_{i}\cdot X_{i}(\varphi)+\frac{1}{\sqrt{2t}}\sum_{1\leq k \leq d_{2}}Y_{k}\cdot  Y_{k}(\phi)$ is the Dirac operator associated to the canonical connection.  The adjoint operator of the Dirac operator  $D^{s, t}$   is given by (see \cite{Wolf, FS})  
  \begin{eqnarray*}
  (D^{s, t})^{*}(\varphi)&=&\sum_{i}\big\{\nabla_{X_{i}}^{s, t}+\Div^{g_{t}}(X_{i})\big\}(\kappa_{n}(X_{i})\varphi) +\sum_{k}\big\{\nabla^{s, t}_{W_{k}}+\Div^{g_{t}}(W_{k})\big\}(\kappa_{n}(Y_{k})\varphi)\\
  &=&\sum_{i}\big\{\nabla_{X_{i}}^{s, t}+\Div^{g_{t}}(X_{i})\big\}(X_{i}\cdot \varphi) +\sum_{k}\big\{\nabla^{s, t}_{W_{k}}+\Div^{g_{t}}(W_{k})\big\}(Y_{k}\cdot\varphi),
 \end{eqnarray*}
where  in a line with the definition of $D^{s, t}$ we interpret  the Clifford multiplication  only in terms of $B$-orthonormal vectors. Since $\nabla^{s, t}$ is metric, 
one has  (see \cite{Baum, Srni})
  \[
\nabla^{s, t}_{X_{i}}(X_{i}\cdot\varphi)=\nabla^{s, t}_{X_{i}}X_{i}\cdot\varphi+X_{i}\cdot\nabla^{s, t}_{X_{i}}\varphi, \   \nabla^{s, t}_{W_{k}}(Y_{k}\cdot\varphi)=\nabla^{s, t}_{W_{k}}Y_{k}\cdot\varphi+Y_{k}\cdot\nabla^{s, t}_{W_{k}}\varphi, \ \  \forall \varphi\in\Gamma(\Sigma). 
  \]
 But by  the definition  of $\nabla^{s, t}$, it follows that   $\nabla^{s, t}_{X_{i}}X_{i}=0=\nabla^{s, t}_{W_{k}}Y_{k}$.  Hence  
  \[
 \big\{D^{s, t}-(D^{s, t})^{*}\big\}(\varphi)=\sum_{1\leq i\leq d_{i}}\Div^{g_{t}}(X_{i})(X_{i}\cdot\varphi)+\frac{1}{\sqrt{2t}}\sum_{1\leq k\leq d_{2}}\Div^{g_{t}}(Y_{k})(Y_{k}\cdot\varphi).
 \]
 However, it is very easy to prove that $\Div^{g_{t}}(X_{i})=0=\Div^{g_{t}}(Y_{k})$ and thus $(D^{s, t})^{*}=D^{s, t}$ for any $s\in\bb{R}$ and $t>0$. We finally mention that an alternative  method available to prove the identification   $(D^{s, t})^{*}=D^{s, t}$, relies on \cite[Prop.~3.1]{Agr03}.
    \end{proof}

  We  describe now the Ricci tensor  and the scalar curvature of  $(M=G/K, \fr{m}_{1}\oplus\fr{m}_{2}, g_{t})$   with respect to the family $\{\nabla^{s, t} : s\in\bb{R}, t>0\}$. 
        Let  $B_{\fr{k}}:=B|_{\fr{k}\times\fr{k}}$  be the restriction of $B$ on the Lie subalgebra $\fr{k}$. 
       \begin{lemma}\label{GEN}
    a) Let $X\in\fr{m}_{1}$ and $Y\in\fr{m}_{1}$. Then, for $Z\in\fr{m}_{1}$  one has
\begin{eqnarray*}
g_{t}(R^{s, t}(X, Z)Z, Y)&=&\frac{(s^{2}t-2s+2st)}{2}B([[X, Z]_{\fr{m}_{2}}, Z], Y)-B_{\fr{k}}([X, Z], [Z, Y]), 
\end{eqnarray*}
while for $Z\in\fr{m}_{2}$ it holds:  
$g_{t}(R^{s, t}(X, Z)Z, Y)=st(s-st-1)B([[X, Z], Z], Y)$. 

\noindent b) Let $X\in\fr{m}_{1}$ and $Y\in\fr{m}_{2}$, or $X\in\fr{m}_{2}$ and $Y\in\fr{m}_{1}$.  Then,  
$g_{t}(R^{s, t}(X, Z)Z, Y)=0$, for any    $Z\in\fr{m}$.

\noindent c) Let $X\in\fr{m}_{2}$ and $Y\in\fr{m}_{2}$. Then  
\begin{eqnarray*}
g_{t}(R^{s, t}(X, Z)Z, Y)&=&st(s-st-1)B([[X, Z], Z]_{\fr{m}_{2}}, 	Y), \quad \text{if} \ Z\in\fr{m}_{1}, \\
g_{t}(R^{s, t}(X, Z)Z, Y)&=&-2t B_{\fr{k}}([X, Z], [Z, Y]), \quad \text{if} \ Z\in\fr{m}_{2}.
\end{eqnarray*}
  \end{lemma}
It is useful to express the splitting   
$\fr{m}=\fr{m}_{1}\oplus\fr{m}_{2}$  by  $\chi_{*}=\chi_{*}^{1}\oplus\chi_{*}^{2}$, where the sub-representations $\chi_{*}^{i} : \fr{k}\to\fr{so}(\fr{m}_{i})$ are given by $\chi_{*}^{i}(Y):=\ad(Y)|_{\fr{m}_{i}}$, for any $Y\in\fr{k}$. Then, for    the Casimir element $C_{\chi, B_{\fr{k}}} : \fr{m}\to\fr{m}$    we write  $C_{\chi}=C_{\chi_{1}}\oplus C_{\chi_{2}}$, where $C_{\chi_{i}} : \fr{m}_{i}\to\fr{m}_{i}$  $(i=1, 2)$
are given by $C_{\chi_{1}}=-\sum_{a=1}^{\dim_{\bb{R}}\fr{k}}\chi^{1}_{*}(k_{a})\circ\chi_{*}^{1}(k'_{a})$ and $C_{\chi_{2}}=-\sum_{a=1}^{\dim_{\bb{R}}\fr{k}}\chi^{2}_{*}(k_{a})\circ\chi_{*}^{2}(k'_{a})$, respectively.  Here, $\{k_{a}, k'_{a}\}$ are dual bases of $\fr{k}$ with respect to $B_{\fr{k}}$.      
   Because   $B$  is the Killing metric,  it is necessarily $C_{\chi_{i}}=\Cas_{i}\cdot \Id_{\fr{m}_{i}}$ with  $\Cas_{i}=B(\lambda_{i}, \lambda_{i}+2\delta)>0$, where $\lambda_{i}$ is  the dominant  weight  of the $K$-module  $\fr{m}_{i}$ and $\delta$ denotes the half of the sum of positive roots of  $\fr{k}\otimes\bb{C}$ \cite{Wa1}.  In other words      (see \cite{Wa1} or \cite[p. 197]{Bes})
   \begin{eqnarray*}
   \Cas_{1}&=&B\big(C_{\chi_{1}}X_{j}, X_{j}\big)=\sum_{i=1}^{d_{1}}B_{\fr{k}}([X_{j},  X_{i}], [X_{j},  X_{i}])=A_{1}(X_{j}, X_{j}),\\
   \Cas_{2}&=&B\big(C_{\chi_{2}}Y_{l}, Y_{l}\big)=\sum_{k=1}^{d_{2}}B_{\fr{k}}([Y_{l}, Y_{k}], [Y_{l},  Y_{k}])=A_{2}(Y_{l}, Y_{l}),
   \end{eqnarray*}
   where   we  define  symmetric bilinear maps $A_{i}$ on $\fr{m}_{i}$ $(i=1, 2)$ by
   \begin{eqnarray*}
   A_{1}(X, Y)&:=&B(C_{\chi_{1}}X, Y)=\sum_{i=1}^{d_{1}}B_{\fr{k}}([X, X_{i}], [Y, X_{i}]), \quad X, Y\in\fr{m}_{1},\\
  A_{2}(X, Y)&:=&B(C_{\chi_{2}}X, Y)=\sum_{k=1}^{d_{2}}B_{\fr{k}}([X, Y_{k}], [Y, Y_{k}]), \quad X, Y\in\fr{m}_{2}. 
   \end{eqnarray*}

  \begin{theorem}\label{rricci}
The Ricci tensor $\Ric^{s, t}$ of the homogeneous Riemannian manifold   $(M=G/K, \fr{m}_{1}\oplus\fr{m}_{2}, g_{t})$ endowed with the family of $G$-invariant metric connections  $\{\nabla^{s, t} : s\in\bb{R}, t\in\bb{R}_{+}\}$, is expressed as follows:

\noindent (a)  Let $X, Y\in\fr{m}_{1}$. Then 
  \begin{eqnarray*}
  \Ric^{s, t}(X, Y) &=& \sum_{i=1}^{d_{1}}\frac{(s^{2}t-2s+2st)}{2}B([[X, X_{i}]_{\fr{m}_{2}}, X_{i}], Y)\\
  &&+\sum_{k=1}^{d_{2}}\frac{(s^{2}-s^{2}t-s)}{2}B([[X, Y_{k}], Y_{k}], Y)+A_{1}(X, Y).
  \end{eqnarray*}
  
\noindent  (b)  Let $X\in\fr{m}_{1}$, $Y\in\fr{m}_{2}$, or $X\in\fr{m}_{2}$, $Y\in\fr{m}_{1}$. Then $\Ric^{s, t}(X, Y)=0$. 
  
\noindent   (c) Let $X, Y\in\fr{m}_{2}$. Then
  \begin{eqnarray*}
  \Ric^{s, t}(X, Y)&=&\sum_{i=1}^{d_{1}}(s^{2}t-s^{2}t^{2}-st)B([[X, X_{i}], X_{i}]_{\fr{m}_{2}}, Y)+A_{2}(X, Y).
  \end{eqnarray*}
    \end{theorem}
     \begin{proof}
 Given some $g_{t}$-orthonormal vectors $V_{i}\in\fr{m}_{1}, W_{k}\in\fr{m}_{2}$, it is $R^{s, t}(X, V_{i})V_{i}=R^{s, t}(X, X_{i})X_{i}$  and  $R^{s, t}(X, W_{k})W_{k}=(1/2t)R^{s, t}(X, Y_{k})Y_{k}$, respectively.  Thus,  for any $X, Y\in\fr{m}$ we compute
\[
 \Ric^{s, t}(X, Y)= \sum_{i=1}^{d_{1}}g_{t}(R^{s, t}(X, X_{i})X_{i}, Y)+(1/2t)\sum_{k=1}^{d_{2}}g_{t}(R^{s, t}(X, Y_{k})Y_{k}, Y) 
\]
 and the result follows by  Lemma \ref{GEN}.  
   \end{proof}

   \begin{corollary}\label{scalarcurv}
 The scalar curvature  ${\rm Scal}^{s, t}$ of $(M=G/K, \fr{m}_{1}\oplus\fr{m}_{2}, g_{t}, \nabla^{s, t})$ is the function on $M$ given by
 \begin{eqnarray*}
 {\rm Scal}^{s, t}&=& -\frac{(s^{2}t-2s+2st)}{2}\sum_{i, j}\|[X_{i}, X_{j}]_{\fr{m}_{2}}\|^{2}-\sum_{i, k}(s^{2}-s^{2}t-s)\|[X_{i}, Y_{k}]\|^{2}\\
 &&+\sum_{i}A_{1}(X_{i}, X_{i}) +\frac{1}{2t}\sum_{k}A_{2}(Y_{k}, Y_{k}).
 \end{eqnarray*}
 \end{corollary}
   \subsection{The homogeneous Einstein equation}
Let us shortly  illustrate the traditional homogeneous Einstein equation $\Ric^{1, t}=cg_{t}$ (where $c\in\bb{R}_{+}$ is the Einstein constant).   We need the components $r_{1}=\Ric^{1, t}(V_{j}, V_{j})$ and  $r_{2}=\Ric^{1, t}(W_{l}, W_{l})$ of the Riemannian Ricci tensor, for some $g_{t}$-orthonormal vectors $V_{j}\in\fr{m}_{1}$ and $W_{l}\in\fr{m}_{2}$, respectively.   Because $\fr{m}_{1}\ncong\fr{m}_{2}$ as $K$-representations, it is  $\Ric^{1, t}(\fr{m}_{1}, \fr{m}_{2})=0$  and all  homogeneous Einstein metrics, if existent, appear as real positive solutions of the equation  $r_{1}-r_{2}=0$.  As a first step,    by  Theorem  \ref{rricci} we see that     \begin{corollary}\label{r1r22}
  Let   $\{X_{i}\}_{i=1}^{d_{1}}$ and $\{Y_{k}\}_{k=1}^{d_{2}}$ be   the $B$-orthonormal bases of $\fr{m}_{1}$ and $\fr{m}_{2}$, respectively. Then, it holds that   $\Ric^{s, t}(X_{i}, Y_{k})=0=\Ric(Y_{k}, X_{i})$, and   
   \begin{eqnarray*}
     \Ric^{s, t}(X_{j},  X_{j})&=&-\displaystyle\frac{(s^{2}t-2s+2st)}{2}\sum_{i=1}^{d_{1}}\| [X_{j},  X_{i}]_{\fr{m}_{2}}\|^{2}-\frac{(s^{2}-s^{2}t-s)}{2}\sum_{k=1}^{d_{2}}\|[X_{j}, Y_{k}]\|^{2}\\
     &&+\Cas_{1},\\
     \Ric^{s, t}(Y_{l}, Y_{l})&=&-st(s-st-1)\sum_{i=1}^{d_{1}}\|[Y_{l}, X_{i}]\|^{2}+\Cas_{2}.
   \end{eqnarray*}
   \end{corollary}
    The Riemannian Ricci tensor occurs for $s=1$, i.e. $r_{1}=\Ric^{1, t}(V_{j}, V_{j})=\Ric^{1, t}(X_{j}, X_{j})$ and $r_{2}=\Ric^{1, t}(W_{l}, W_{l})=(1/2t)\Ric^{1, t}(Y_{l}, Y_{l})$, respectively. 
Since $t\neq 0$, we conclude that  the homogeneous Einstein equation is   the quadratic equation  $\al\cdot t^{2}  + \be\cdot  t  +\gamma=0$, where $\al, \be, \gamma$ are given by
\begin{eqnarray*}
\al&:=&-3\sum_{i}\|[X_{j}, X_{i}]_{\fr{m}_{2}}\|^{2}+\sum_{k}\|[X_{j}, Y_{k}]\|^{2}-\sum_{i}\|[Y_{l}, X_{i}]\|^{2},\\
\be&:=&2\Big(\sum_{i}\|[X_{j}, X_{i}]_{\fr{m}_{2}}\|^{2}+ \Cas_{1}\Big),\quad \gamma:=-\Cas_{2}.
\end{eqnarray*}
Because the sign of $\al$ depends on the underlying manifold, all that we can state is a bound of the number of invariant Einstein metrics $\fr{ein}(M)$, namely $0\leq \fr{ein}(M) \leq 2$, see also \cite{Kerr} and  \cite[p. 263-264]{Bes}.  
 
\begin{example}\label{cp3}\textnormal{Consider the  complex projective space $\bb{C}P^{3}=\SO(5)/\U(2)$.    For the Lie algebra $\fr{so}(5)$   we fix  a reductive decomposition related to the twistor fibration of $\bb{C}P^{3}$ over the 4-sphere $\Ss^{4}$ (for details we refer to \cite{Baum}).  Recall that the matrices  $\{E_{i, j} :  i<j\}$  form  an orthonormal basis of $\fr{so}(n)$ with respect to the scalar product $B'=-(1/2)\tr AB$  (which is such that $B_{\SO(n)}=2(n-2)B'$).  By definition,  it is  $\fr{so}(5)=\Span_{\bb{R}}\{E_{1, 2}, \ldots,  E_{4, 5}\}$. Set
  {\small\[
  \fr{k}:=\Span_{\bb{R}}\{k_{1}:=E_{1, 2}, \ k_{2}:=E_{3, 4}, \  k_{3}:=(E_{1, 3}-E_{2, 4})/\sqrt{2},  \ k_{4}:=(E_{1, 4}+E_{2, 3})/\sqrt{2}\} \cong\fr{u}(2).
\]}
  Notice  that  $B'(k_{i}, k_{j})=\delta_{i, j}$, for any $1\leq i, j\leq 4$.  
         Let $\fr{m}$ be  the invariant $B'$-orthogonal complement of $\fr{k}$  into $\fr{so}(5)$;  an orthonormal   basis  with respect to $B'$  is given by the vectors $e_{1}:=E_{1, 5}$, $e_{2}:=E_{2, 5}$,  $e_{3}:=E_{3, 5}$,  $e_{4}:=E_{4, 5}$,  $e_{5}:=(E_{1, 3}+E_{2, 4})/\sqrt{2}$, and  $e_{6}:=(E_{1, 4}-E_{2, 3})/\sqrt{2}$. 
We also set 
\[
\fr{m}_{1}:=\Span_{\bb{R}}\{e_{1}, e_{2}, e_{3}, e_{4}\}, \quad \fr{m}_{2}:=\Span_{\bb{R}}\{e_{5}, e_{6}\},
\]
 such that $\fr{m}=\fr{m}_{1}\oplus\fr{m}_{2}$.  Then,  the inclusions given by (\ref{incl}) can be easily checked by computing  the      Lie brackets  $\ad_{i,j}:=\ad(e_{i})e_{j}=[e_{i}, e_{j}]$ $(1\leq i, j\leq 6)$ of   the base vectors: 
 \[
\begin{tabular}{r|cccccc}
& $\ad_{1, j}$ & $\ad_{2, j}$ & $\ad_{3, j}$ & $\ad_{4, j}$ & $\ad_{5, j}$ & $\ad_{6, j}$ \\
\hline
$e_{1}$ &  0 & $-k_{1}$ & $-E_{1, 3}$  & $-E_{1, 4}$ & $(\sqrt{2}/{2})e_{3}$ & $(\sqrt{2}/{2})e_{4}$ \\
\hline
$e_{2}$ &  $k_{1}$ & 0 &  $-E_{2, 3}$ &  $-E_{2, 4}$  & $(\sqrt{2}/{2})e_{4}$  &  $-(\sqrt{2}/{2})e_{3}$    \\
\hline
$e_{3}$ & $E_{1, 3}$ &  $E_{2, 3}$ & 0 & $-k_{2}$ & $-(\sqrt{2}/{2})e_{1}$ &  $(\sqrt{2}/{2})e_{2}$   \\
\hline
$e_{4}$ & $E_{1, 4}$ &  $E_{2, 4}$   & $k_{2}$ &  0 & $-(\sqrt{2}/{2})e_{2}$ &  $-(\sqrt{2}/{2})e_{1}$  \\
\hline
$e_{5}$ & $-(\sqrt{2}/{2})e_{3}$ &  $-(\sqrt{2}/{2})e_{4}$ &   $(\sqrt{2}/{2})e_{1}$ &  $(\sqrt{2}/{2})e_{2}$ & 0 & $ k_{1}-k_{2}$ \\
\hline
$e_{6}$ & $-(\sqrt{2}/{2})e_{4}$ &  $(\sqrt{2}/{2})e_{3}$ & $-(\sqrt{2}/{2})e_{2}$ & $(\sqrt{2}/{2})e_{1}$ & $-k_{1}+k_{2}$ & 0  \\
\hline
\end{tabular}
\]
Observe that  $E_{1, 3}, E_{1,4}, E_{2, 3}, E_{2, 4}\in \fr{k}\oplus\fr{m}_{2}$.  The restrictions   separately on $\fr{k}$ and $\fr{m}_{2}$,  are given by 
  \begin{equation*}\label{sorry}
  \begin{tabular}{l  l  l  l }
 $E_{1, 3}|_{\fr{m}_{2}}= (\sqrt{2}/{2})e_{5},$  & $E_{1, 4}|_{\fr{m}_{2}}=(\sqrt{2}/{2})e_{6},$  & $E_{2, 3}|_{\fr{m}_{2}}= -(\sqrt{2}/{2})e_{6},$ & $E_{2, 4}|_{\fr{m}_{2}}=(\sqrt{2}/{2})e_{5}$, \\
 $E_{1, 3}|_{\fr{k}}=(\sqrt{2}/{2})k_{3},$      &  $E_{1, 4}|_{\fr{k}}= (\sqrt{2}/{2})k_{4},$ &   $E_{2, 3}|_{\fr{k}}= (\sqrt{2}/{2})k_{4},$   &  $E_{2, 4}|_{\fr{k}}=-(\sqrt{2}/{2})k_{3}$.  
  \end{tabular} 
  \end{equation*} 
   Now, up to scaling, any invariant Riemannian metric on $\bb{C}P^{3}$  has the form  $g_{t}=B'|_{\fr{m}_{1}\times\fr{m}_{1}}+2tB'|_{\fr{m}_{2}\times\fr{m}_{2}}$, for some $t\in\bb{R}_{+}$.  For $X_{j}=e_{1}$ and $Y_{l}=e_{5}$, the coefficients $\al, \be, \gamma$ are given respectively by  $\al=-3\sum_{i=1}^{4}\big\| [e_{1}, e_{i}]_{\fr{m}_{2}}\big\|^{2}+\sum_{k=5}^{6}\big\|[e_{1}, e_{k}]\big\|^{2}-\sum_{i=1}^{4}\big\|[e_{5},  e_{i}]\big\|^{2}$, $\be= 2\Big(\sum_{i=1}^{4}\big\|[e_{1}, e_{i}]_{\fr{m}_{2}}\big\|^{2}+ B'\big(C_{\chi_{1}}e_{1}, e_{1}\big)\Big)$ and  $\gamma=-B'\big(C_{\chi_{2}}e_{5}, e_{5}\big)$.   We  compute  $\gamma=-B'\big(C_{\chi_{2}}e_{5}, e_{5}\big)=-2(=-B'\big(C_{\chi_{1}}e_{1}, e_{1}\big))$, $\al=-4$  and $\be=6$. Therefore,  on $\bb{C}P^{3}$ the Einstein equation $\al\cdot t^{2}+\be\cdot t+\gamma=0$ has two positive solutions, namely $t=1$ and $t=1/2$.  The first value defines the K\"ahler-Einstein metric  $g_{1}=B'|_{\fr{m}_{1}\times\fr{m}_{1}}+2B'|_{\fr{m}_{2}\times\fr{m}_{2}}$ and the second one corresponds to the Killing metric $g_{1/2}=B'|_{\fr{m}_{1}\times\fr{m}_{1}}+B'|_{\fr{m}_{2}\times\fr{m}_{2}}$, which is a homogeneous Einstein metric for $\bb{C}P^{3}$ (in fact nearly-K\"ahler), see  also \cite{Baum, Chry1}.}
 \end{example}

 \subsection{$\nabla^{s, \frac{1}{2}}$-Einstein structures}
Corollary  \ref{NEWTORSION} ensures that  for $s\neq 1$ and $t=1/2$  the family $\{\nabla^{s, \frac{1}{2}} : s\in\bb{R}\}$ has non-trivial skew-symmetric torsion  $T^{s, \frac{1}{2}}\in\Lambda^{3}(\fr{m})$. Thus,   if $\nabla^{g_{1/2}}\equiv\nabla^{1, \frac{1}{2}}$ denotes the Levi-Civita connection on the standard homogeneous Riemannian  manifold  $(M=G/K, \fr{m}_{1}\oplus\fr{m}_{2}, g_{1/2})$, one can write $B(\nabla^{s}_{X}Y, Z)=B(\nabla^{g_{1/2}}_{X}Y, Z) +\frac{1}{2}T^{s}(X, Y, Z)$.    Hence, for the Killing metric $g_{1/2}\equiv B|_{\fr{m}\times\fr{m}}$ it makes sense to examine the existence of a  $\nabla^{s, \frac{1}{2}}$-Einstein structure with skew-torsion.  Because  the value $t=1/2$ will be fixed from now on,  for simplicity we will write $\nabla^{s, \frac{1}{2}}\equiv \nabla^{s}$, $T^{s, \frac{1}{2}}\equiv T^{s}$,  e.t.c.
  Let us present   some structural properties of the torsion form $T^{s}\in\Lambda^{3}(\fr{m})$.  
It is useful to introduce the following  maps:
   \[
 \begin{tabular}{l}
  $ \Jac_{\fr{m}}(X_{j}, X_{r}, X_{s}):= \ \fr{S}_{j, r, s}[ X_{j},  [X_{r}, X_{s}]_{\fr{m}_{2}}]$,\\
  $\Jac_{\fr{m}}(X_{i}, X_{j},  Y_{k})   \ := \ [X_{i},  [X_{j}, Y_{k}]]_{\fr{m}_{2}}+[X_{j},   [Y_{k}, X_{i}]]_{\fr{m}_{2}}+[Y_{k},  [X_{i}, X_{j}]_{\fr{k}}]$,\\
  $\Jac_{\fr{m}}(X_{i}, Y_{k}, Y_{l}) \ \  := \ \fr{S}_{i, k, l}[X_{i}, [Y_{k}, Y_{l}]].$
    \end{tabular}
 \]
 Here, the vectors $\{X_{i}\}_{i=1}^{d_{1}}$ and $\{Y_{k}\}_{k=1}^{d_{2}}$ stand for the fixed $B$-orthonormal bases.   Although these  maps are different each other, we use   the same notation since  in any case their definition is obvious by  (\ref{incl}).  
  In fact, it is easy to see that the mixed  Jacobians $\Jac_{\fr{m}}(X_{i}, X_{j}, Y_{k})$ and $\Jac_{\fr{m}}(X_{i}, Y_{k}, Y_{l})$  are identically    equal to zero.     Indeed, by  viewing  the vectors $X_{i}, Y_{k}, Y_{l}\in\fr{m}$ as left-invariant vectors fields, it follows that $\Jac_{\fr{m}}(Y_{k}, Y_{l}, X_{i})=0$.  Then,  because $B\big(\Jac_{\fr{m}}(X_{i}, X_{j}, Y_{k}), Y_{l}\big)=B\big(\Jac_{\fr{m}}(X_{i}, Y_{k}, Y_{l}), X_{j}\big)$  it also follows that   $\Jac_{\fr{m}}(X_{i}, X_{j}, Y_{k})=0$. This observation simplifies the calculations.
In the following, for some   vectors $A\in\fr{m}_{\al}, B\in\fr{m}_{\be}, C\in\fr{m}_{\gamma}$ with $1\leq \al, \be, \gamma \leq 2$, we shall use the convention $\nabla_{\al\be\gamma}^{s}:=(\nabla_{C}^{s}T^{s})(A, B)$.  
\begin{lemma}\label{COVAR}
 The 
 covariant derivative  $\nabla_{Z}^{s}T^{s} : \fr{m}\times\fr{m}\to\fr{m}$ of the torsion form $T^{s, \frac{1}{2}}\equiv T^{s}$ is given by  \[
 \begin{tabular}{lllll}
$\nabla_{111}^{s}$&:=&$(\nabla_{X_{r}}^{s}T^{s})(X_{i}, X_{j})$ &=&$\frac{s(s-1)}{2}\Jac_{\fr{m}}(X_{i}, X_{j}, X_{r})$,\\
$\nabla_{112}^{s}$&:=&$(\nabla_{Y_{k}}^{s}T^{s})(X_{i}, X_{j})$&=& $-\frac{s(s-1)}{2}[Y_{k}, [X_{i}, X_{j}]_{\fr{k}}]$,\\
$\nabla_{221}^{s}$&:=&$(\nabla_{X_{i}}^{s}T^{s})(Y_{k}, Y_{l})$&=& $-\frac{s(s-1)}{2}[X_{i}, [Y_{k}, Y_{l}]]$,
\end{tabular}
\]
with $\nabla^{s}_{112}=-\nabla^{s}_{121}=\nabla^{s}_{211}$, $\nabla^{s}_{221}=-\nabla^{s}_{212}=\nabla^{s}_{122}$ and  all the other combinations are zero.    On the other hand, the co-differential $\delta^{s}T^{s}$  vanishes  for any $s\in\bb{R}$.
 \end{lemma}

\noindent The vanishing of the co-differential $\delta^{s}T^{s}$ ensures that the Ricci tensor $\Ric^{s}$ is symmetric.  In full details
\begin{proposition}
Let   $\{X_{i}\}_{i=1}^{d_{1}}$ and $\{Y_{k}\}_{k=1}^{d_{2}}$ be   the $B$-orthonormal bases of $\fr{m}_{1}$ and $\fr{m}_{2}$, respectively. Then, the Ricci tensor  associated to the  1-parameter family   $\{\nabla^{s}\equiv\nabla^{s, \frac{1}{2}} : s\in\bb{R}\}$ satisfies the following relations: $\Ric^{s}(X_{i}, Y_{l})=0$ and
\[
 \Ric^{s}(X_{j},  X_{j})=\Ric^{g_{1/2}}(X_{j},  X_{j})-\frac{1}{4}S^{s}(X_{j}, X_{j}),\ \  \Ric^{s}(Y_{l}, Y_{l})=\Ric^{g_{1/2}}(Y_{l}, Y_{l})-\frac{1}{4}S^{s}(Y_{l}, Y_{l}),
\]
  where the non-zero parts of the  symmetric tensor $S^{s}$ are of the form
\[
S^{s}(X_{j}, X_{j})=(s-1)^{2}\big\{\sum_{i}\|[X_{j}, X_{i}]_{\fr{m}_{2}}\|^{2}+\sum_{k}\|[X_{j}, Y_{k}]\|^{2}\big\}, \  \ S^{s}(Y_{l}, Y_{l})=(s-1)^{2}\sum_{i}\|[Y_{l}, X_{i}]\|^{2}.
\]
\end{proposition}
\begin{proof}
By Corollary \ref{r1r22}, the Ricci tensor $\Ric^{s}\equiv\Ric^{s, \frac{1}{2}}$ with respect to  $\nabla^{s}$ has the form
 \begin{eqnarray*}
  \Ric^{s}(X_{j},  X_{j})&=&-\displaystyle\frac{(s^{2}-2s)}{4}\Big\{\sum_{i=1}^{d_{1}}\| [X_{j},  X_{i}]_{\fr{m}_{2}}\|^{2}+\sum_{k=1}^{d_{2}}\|[X_{j}, Y_{k}]\|^{2}\Big\}+\Cas_{1},\\
     \Ric^{s}(Y_{l}, Y_{l})&=&-\displaystyle\frac{(s^{2}-2s)}{4}\sum_{i=1}^{d_{1}}\|[Y_{l}, X_{i}]\|^{2}+\Cas_{2}.
 \end{eqnarray*}
Using now the Riemannian Ricci tensor $\Ric^{1, \frac{1}{2}}\equiv\Ric^{1}\equiv \Ric^{g_{1/2}}$ and the definition of the symmetric tensor $S^{s}$, one can obtain the given expressions. For example, let $0\neq X, Y\in\fr{m}_{1}$. Then
\begin{eqnarray*}
S^{s}(X, Y)&:=&\sum_{j=1}^{d_{1}}B\big(T^{s}(X_{j}, X), T^{s}(X_{j}, Y)\big)+\sum_{l=1}^{d_{2}}B\big(T^{s}(Y_{l}, X), T^{s}(Y_{l}, Y)\big)\\
&=&(s-1)^{2}\Big\{\sum_{j}B([X_{j}, X]_{\fr{m}_{2}}, [X_{j}, X]_{\fr{m}_{2}})+\sum_{l}B([Y_{l}, X], [Y_{l}, Y])\Big\}.
\end{eqnarray*}
Similarly, for $X, Y\in\fr{m}_{2}$ we get $S^{s}(X, Y)=\sum_{j}B\big(T^{s}(X_{j}, X), T^{s}(X_{j}, X)\big)=(s-1)^{2}\sum_{j}B([X_{j}, X], [X_{j}, Y])$, and finally for $X\in\fr{m}_{1}, Y\in\fr{m}_{2}$ it is $S^{s}(X, Y)=0=S^{s}(Y, X)$.  
\end{proof}

 \begin{theorem}\label{THM5}
 Let $(M=G/K, \fr{m}_{1}\oplus\fr{m}_{2}, g_{1/2})$ be a compact connected homogeneous Riemannian manifold  with two isotropy summands satisfying  (\ref{incl}). Then, $M=G/K$ is a $\nabla^{s}$-Einstein manifold with skew-torsion $0\neq T^{s}\in\Lambda^{3}(\fr{m})$ for the values $s=0$ or $s=2$,  if and only if,   the Killing metric $g_{B}\equiv g_{1/2}$ is a $G$-invariant Einstein metric, i.e. $\Cas_1=\Cas_2$. 
   \end{theorem}
   \begin{proof}
   A $\nabla^{s}$-Einstein structure on $(M=G/K, \fr{m}_{1}\oplus\fr{m}_{2}, g_{1/2})$ is given as a solution (with respect to   $s$) of the following system:
\[
\Big\{ \Ric^{s}(X_{j},  X_{j})=\frac{\Sca^{s}}{n}B(X_{j}, X_{j})=\frac{\Sca^{s}}{n}, \ \  \Ric^{s}(Y_{l}, Y_{l})=\frac{\Sca^{s}}{n}B(Y_{l}, Y_{l})=\frac{\Sca^{s}}{n}\Big\}.
\]
This is equivalent  to the equation $ \Ric^{s}(X_{j},  X_{j})-\Ric^{s}(Y_{l}, Y_{l})=0$, namely
\[
 (s^{2}-2s)\Big\{\sum_{i=1}^{d_{1}}\| [X_{j},  X_{i}]_{\fr{m}_{2}}\|^{2}+\sum_{k=1}^{d_{2}}\|[X_{j}, Y_{k}]\|^{2}-\sum_{i=1}^{d_{1}}\|[Y_{l}, X_{i}]\|^{2}\Big\}=4(\Cas_{1}-\Cas_{2}).
\]
Since $\fr{m}=\fr{m}_{1}\oplus\fr{m}_{2}$ is a $B$-orthogonal decomposition of  $\fr{m}=T_{o}G/K$ and both $\fr{m}_{1}, \fr{m}_{2}$ have been assumed to be irreducible and non-equivalent, by   \cite[Thm. 1.11]{Wa1} it is known that  $M=G/K$ is a standard  homogeneous Einstein manifold  if and only if  the Casimir constants coincide $\Cas_{1}=\Cas_{2}$, and our assertion follows.
\end{proof}

 Recall that $T^{s}(X, Y)=(s-1)[X, Y]_{\fr{m}}$, hence the torsion of  the  connections $\nabla^{0}$ and $\nabla^{2}$ are such that $T^{0}=-T^{2}$.  Although $\nabla^{0}T^{0}=0$, Lemma \ref{COVAR} ensures that $\nabla^{2}T^{2}\neq 0$ (in analogy to the isotropy irreducible case).  

 \begin{remark}\textnormal{The $\nabla^{s}$-Einstein condition on $(M=G/K, g_{1/2})$ is  the following   quadratic equation:
 \[
 \fr{c}s^{2}-2\fr{c}s-4(\Cas_1-\Cas_2)=0,
 \]
where $\fr{c}:=(\sum_{i=1}^{d_{1}}\| [X_{j},  X_{i}]_{\fr{m}_{2}}\|^{2}+\sum_{k=1}^{d_{2}}\|[X_{j}, Y_{k}]\|^{2}-\sum_{i=1}^{d_{1}}\|[Y_{l}, X_{i}]\|^{2})$.  The  discriminant  is given by $\Delta=4\fr{c}(\fr{c}+4(\Cas_{1}-\Cas_{2}))$.  Obviously, if  $\fr{c}>0$ and $\Cas_{1}\geq \Cas_{2}$, then $\Delta>0$ and there are two solutions, i.e. two $\nabla^{s}$-Einstein structures, defined by
\[
s=1\pm \sqrt{\frac{\fr{c}+4(\Cas_{1}-\Cas_{2})}{\fr{c}}}.
\]
  The same is true if   $\fr{c}<0$ and $\Cas_{2}\geq \Cas_1$.  Assuming that  the Killing metric $g_{B}$ is Einstein, then we recover the values $s=0, 2$ described in Theorem \ref{THM5}. In general, if $\fr{c}>0$ with $\fr{c}>4(\Cas_{2}-\Cas_{1})$, or $\fr{c}<0$ with $\fr{c}<4(\Cas_{1}-\Cas_{2})$, then $\Delta>0$ and the  $\nabla^{s}$-Einstein structures described above are still available. Hence,   one can theoretically describe    $\nabla^{s}$-Einstein structures which are different than the $\nabla^{s}$-Einstein structures associated to the canonical $(s=0)$ and the anti-canonical connection $(s=2)$. The same time, the case    $\Delta<0$ is still possible. In fact, if $\Cas_{1}\neq \Cas_{2}$, only such examples we are able to construct, i.e.  cosets $(M=G/K, g_{1/2})$ with two isotropy summands satisfying (\ref{incl}), which  do not admit  any $\nabla^{s}$-Einstein structure. Explicit examples of cosets $(M=G/K, g_{1/2})$ carrying $\nabla^{s}$-Einstein structures with $s\neq 0, 2$ are still missing; this  interesting topic  will be addressed in a forthcoming work.} 
 \end{remark}

 \subsection{Examples}
 Let us present now a series of manifolds  that Theorem \ref{THM5} can be applied.  We focus   on flag manifolds  and we prove that several of them carry $\nabla^{s}$-Einstein structures for $s=0, 2$.   Let us a fix a   compact {\it simple} Lie group $G$ and let $M=G/K$ be   a  flag manifold  with two isotropy summands, say $\fr{m}=\fr{m}_{1}\oplus\fr{m}_{2}$.  Such spaces have been classified in terms of painted Dynkin diagrams in \cite{Chry}.  Since both $\fr{m}_{1}$ and $\fr{m}_{2}$ are   irreducible and inequivalent, any $G$-invariant Riemannian metric on $M=G/K$   is a multiple of $g_{t}$.   In  \cite[Thm 1.1]{Chry1} it was shown  that $M$ admits precisely two $G$-invariant Einstein metrics; one of them is K\"ahler and corresponds to the value
$t=1$; the  other one is given by $t=\frac{2d_{2}}{d_{1}+2d_{2}}$, i.e. $g_{\frac{2d_{2}}{d_{1}+2d_{2}}}=B|_{\fr{m}_{1}\times\fr{m}_{1}}+\frac{4d_{2}}{d_{1}+2d_{2}}B|_{\fr{m}_{2}\times\fr{m}_{2}}$. 
\begin{theorem}\label{F}\textnormal{(\cite{Chry1})}
Let $G$ be a compact connected simple Lie group.  A generalized flag manifold $M=G/K$  whose isotropy representation is such that $\fr{m}=\fr{m}_{1}\oplus\fr{m}_{2}$, is a standard homogeneous Einstein manifold  if and only if $d_{1}=2d_{2}$, where $d_{i}=\dim\fr{m}_{i}$ for $i=1, 2$.
\end{theorem}

\begin{example}
\textnormal{Consider   the complex projective space  $\bb{C}P^{3}=\SO(5)/\U(2)=\Sp(2)/\Sp(1)\times \U(1)$. It is  $d_{1}=4=2d_{2}$, hence  $\bb{C}P^{3}$ is standard Einstein, i.e. $\Cas_{1}=\Cas_{2}$, see  also Example \ref{cp3}.   Thus,  $(\bb{C}P^{3}, g_{1/2})$ admits  exactly two $\nabla^{s, \frac{1}{2}}$-Einstein structures with skew-torsion, namely these which occurs for  $s=0, 2$. 
      The $\nabla^{0, \frac{1}{2}}$-Einstein structure  is   well-known; it is related with the homogeneous nearly-K\"ahler structure that $(\bb{C}P^{3}, g_{1/2})$ admits, in particular the canonical connection $\nabla^{0, \frac{1}{2}}$  coincides with the characteristic connection $\nabla^{c}$ (Gray connection), see \cite{AFer, Srni}. }
      \end{example}

 A quick check of the dimensions of the isotropy summands implies that  there are no exceptional flag manifolds, with $\fr{m}=\fr{m}_{1}\oplus\fr{m}_{2}$, for which the Killing metric can be a $G$-invariant  Einstein metric (see \cite[p.~245] {Chry1}).   However, several examples appear for adjoint orbits corresponding to  the classical Lie groups $B_{\ell}=\SO(2\ell+1)$, $C_{\ell}=\Sp(\ell)$, or $D_{\ell}=\SO(2\ell)$.   
\begin{example}
\textnormal{For the family $B(\ell, p):=\SO(2\ell+1)/\U(p)\times \SO(2(\ell-p)+1)$ $(2\leq p\leq \ell, \  \ell\geq 2)$  we compute $d_{1}=4p(\ell-p)+2p$ and $d_{2}=p(p-1)$. According to Theorem \ref{F}, the  Killing metric $g_{B}$ is Einstein  if and only if $p=2(\ell+1)/3\in\bb{Z}_{+}$.  Hence  we conclude that the manifold  $B(\ell,  2(\ell+1)/3)=\SO(2\ell+1)/U(2(\ell+1)/3)\times \SO(2(\ell-2)/{3}+1)$, with $\ell=2+3k$ and $k=0, 1, 2, 3, \ldots$,   is a $\nabla^{s}$-Einstein manifold for the values $s=0, 2$.   Let us list  the first examples:
 \[
 \begin{tabular}{c|c|c}
  $\ell$ & $p=2(\ell+1)/3 : p\in\bb{Z}_{+}$ & $(M=G/K, \fr{m}_{1}\oplus\fr{m}_{2}, g_{B})$   \\
  \hline
   $\ell=2$ & $p=2$ & $\bb{C}P^{3}=\SO(5)/\U(2)$ \\
   $\ell=5$ & $p=4$ & $\SO(11)/\U(4)\times\SO(3)$ \\
   $\ell=8$ & $p=6$ & $\SO(17)/\U(6)\times\SO(5)$ \\
   $\vdots$ & $\vdots$ & $\vdots$
   \end{tabular}
 \]
}
\end{example}
  \begin{example}
\textnormal{For the space $C(\ell, p):=\Sp(\ell)/\U(p)\times\Sp(\ell-p)$   $(1\leq p\leq \ell-1, \  \ell\geq 2)$  it is  $d_{1}=4p(\ell-p)$ and $d_{2}=p(p+1)$ and  the condition $d_{1}=2d_{2}$ takes the form  $p=(2\ell-1)/3\in\bb{Z}_{+}$.  Thus the   family $C(\ell, (2\ell-1)/3)=\Sp(\ell)/U((2\ell-1)/{3})\times \Sp((\ell+1)/{3})$, with $\ell=2+3k$ and $k=0, 1, 2, 3, \ldots$, is a standard homogeneous Einstein manifold. Moreover,  for $s=0, 2$ it  becomes a $\nabla^{s}$-Einstein manifold with skew-torsion.    \[
 \begin{tabular}{c|c|c}
  $\ell$ & $p=(2\ell-1)/3 : p\in\bb{Z}_{+}$ & $(M=G/K, \fr{m}_{1}\oplus\fr{m}_{2}, g_{B})$  \\
  \hline
   $\ell=2$ & $p=1$ & $\bb{C}P^{3}=\Sp(2)/\U(1)\times\Sp(1)$ \\
   $\ell=5$ & $p=3$ & $\Sp(5)/\U(3)\times\Sp(2)$ \\
   $\ell=8$ & $p=5$ & $\Sp(8)/\U(5)\times\Sp(3)$ \\
   $\vdots$ & $\vdots$ & $\vdots$
   \end{tabular}
 \]
}
\end{example}
 \begin{example}
\textnormal{For the flag manifold  $D(\ell, p):=\SO(2\ell)/\U(p)\times\SO(2(\ell-p))$   $(2\leq p\leq\ell-2,  \ \ell\geq 4)$ it is  $d_{1}=4p(\ell-p)$ and $d_{2}=p(p-1)$. Hence $D(\ell, p)$ is a  standard  homogeneous Einstein manifold  if and only if   $p=(2\ell+1)/3\in\bb{Z}_{+}$. 
It follows that the family $D(\ell, (2\ell+1)/3)=\SO(2\ell)/U((2\ell+1)/{3})\times \SO(2(\ell-1)/{3})$, with $\ell=4+3k$ and $k=0, 1, 2, 3, \ldots$, 
 admits   $\nabla^{s}$-Einstein structures with skew-torsion  for the values $s=0, 2$.  
    \[
 \begin{tabular}{c|c|c}
  $\ell$ & $p=(2\ell+1)/3 : p\in\bb{Z}_{+}$ & $(M=G/K, \fr{m}_{1}\oplus\fr{m}_{2}, g_{B})$    \\
  \hline
   $\ell=4$ & $p=3$ & $\SO(8)/\U(3)\times\SO(2)$ \\
   $\ell=7$ & $p=5$ & $\SO(14)/\U(5)\times\SO(4)$ \\
   $\ell=10$ & $p=7$ & $\SO(20)/\U(7)\times\SO(6)$ \\
   $\vdots$ & $\vdots$ & $\vdots$
   \end{tabular}
 \]
}
\end{example}

\end{document}